\documentclass[hidelinks,onefignum,onetabnum]{siamart220329}
\usepackage[normalem]{ulem}

\usepackage{lipsum}
\usepackage{amsfonts}
\usepackage{graphicx}
\usepackage{epstopdf}
\usepackage{algorithmic}
\ifpdf
  \DeclareGraphicsExtensions{.eps,.pdf,.png,.jpg}
\else
  \DeclareGraphicsExtensions{.eps}
\fi

\newsiamremark{remark}{Remark}
\newsiamremark{hypothesis}{Hypothesis}
\crefname{hypothesis}{Hypothesis}{Hypotheses}
\newsiamthm{claim}{Claim}

\headers{Continuity conditions weaker than lower semi-continuity}{J Westerhout, X Guo\and H D Nguyen}

\title{Continuity conditions weaker than lower semi-continuity\thanks{\funding{This work was funded by funding Australian Research Council grant DP230100905.}}}

\author{Jacob Westerhout\thanks{School of Mathematics and Physics, University of Queensland, Brisbane,
Australia 
  (\email{j.westerhout@uq.edu.au}).}
\and Xin Guo\thanks{School of Mathematics and Physics, University of Queensland, Brisbane,
Australia 
  (\email{xin.guo@uq.edu.au}).}
\and Hien Duy Nguyen\thanks{Department of Mathematics and Physical Science, La Trobe University,
Melbourne, Australia. Institute of Mathematics for Industry, Kyushu University, Fukuoka, Japan. 
  (\email{hien@imi.kyushu-u.ac.jp}).}}

\usepackage{amsopn}

\usepackage{ amssymb }
\usepackage{ dsfont }
\usepackage{ comment }

\newcommand{\set}[1]{\left\{ #1 \right \} }

\newcommand{\lev}{\mathrm{lev}}

\newcommand{\ab}[1]{\left(#1\right)}

\newcommand{\inprod}[2]{\left\langle#1,#2 \right\rangle}
\newcommand{\argmin}{\mathop{\mathrm{arg\,min}}}

\usepackage{bold-extra}

\usepackage{tikz}

\ifpdf
\hypersetup{
  pdftitle={Generalized Continuity},
  pdfauthor={Me}
}
\fi

\begin{document}

\maketitle
\begin{abstract}
Lower semi-continuity (\texttt{LSC}) is a critical assumption in many foundational optimisation theory results; however, in many cases, \texttt{LSC} is stronger than necessary. This has led to the introduction of numerous weaker continuity conditions that enable more general theorem statements. In the context of unstructured optimization over topological domains, we collect these continuity conditions from disparate sources and review their applications. As primary outcomes, we prove two comprehensive implication diagrams that establish novel connections between the reviewed conditions. In doing so, we also introduce previously missing continuity conditions and provide new counterexamples.
\end{abstract}



\section{Introduction}

When solving optimization problems, lower semi-continuity (\texttt{LSC}) often appears as a pivotal assumption, being a common condition in various versions of the Weierstrass extreme value theorem, Berge maximum theorem, Takahashi minimization theorem, Ky Fan minimax inequality, von Neumann minimax theorem, and Ekeland variational principle, among other essential results in optimization theory, variation analysis, nonlinear analysis, and related areas; see, e.g.,  \cite{attouch2014variational,aubin2013optima,kristaly2010variational,rockafellar2009variational}.  These results typically rely on desirable properties of \texttt{LSC} functions, particularly regarding the set of minimizers, although \texttt{LSC} is often a stronger condition than necessary.

Over time, many relaxations and variations of the \texttt{LSC} condition have been proposed, which seek to retain the salient properties of \texttt{LSC} functions while discarding superfluous ones. In this work, we comprehensively review such weaker continuity conditions and provide examples of their usage. Furthermore, we map the implications and connections between these continuity conditions, some of which are novel and have not previously been reported.

In particular, we focus on weaker notions of the \texttt{LSC} condition for functions mapping from some topological space $\cal{X}$ to the extended reals $\overline{\mathbb R}$ in the context of unstructured optimization problems. We further categorize the considered conditions into three broad categories:
\begin{enumerate}
    \item Relaxations of the definition of \texttt{LSC} functions.
    \item Formalization of a notion of `\texttt{LSC} from above'.
    \item Distillations of the \texttt{LSC} assumption.
\end{enumerate}
We categorize the considered conditions, review their applications, and provide definitions within a unified notation. Note that our survey is by no means exhaustive, as even more notions of continuity are available when $\cal{X}$ is endowed with additional structure (such as vector or metric space conditions), or when the co-domain is an ordered space other than $\overline{\mathbb R}$.

\subsection{Relaxations of \texttt{LSC} functions}
Three main types of continuity conditions relax the definition of \texttt{LSC}: lower quasi-continuous (\texttt{LQC}), introduced in \cite{scalzo2009uniform}; lower pseudo-continuous (\texttt{LPC}), introduced in \cite{morgan2007pseudocontinuous}; and weakly lower continuous (\texttt{WLC}), introduced in \cite{campbell1990optimization,campbell1990optimization_bad}. In the literature, \texttt{LQC} and \texttt{LPC} functions have also been referred to as weakly lower pseudo-continuous, and lower P-continuous, respectively (cf. \cite[def 1]{qiu4811574existence} and \cite[def 4]{scalzo2013essential}).

The most popular of these notions appear to be the \texttt{LPC} functions, which have been studied in the works of \cite{cotrina2022remarks,quartieri2022existence,scalzo2008pseudocontinuity}, and used in the analysis of preference relations \cite{morgan2007pseudocontinuous}, extensions of the Berge maximum theorem \cite{cotrina2022remarks,morgan2007pseudocontinuous, xiaoling2017berge}, existence results for Nash equilibria \cite{morgan2004pseudocontinuity,morgan2007pseudocontinuous,morgan2008variational} and Berge equilibria \cite{deghdak2013existence}, analysis of continuity \cite{anh2016continuity} and stability \cite{zhou2015characterization} of minima, characterizations of order-preserving continuous representations \cite{scalzo2008pseudocontinuity}, extensions of the Weierstrass extreme value theorem \cite{morgan2007pseudocontinuous, amini2016some}, extension of the Ky Fan minimax inequality \cite{cotrina2022remarks}, and to characterize solutions of minimax problems \cite{morgan2004pseudocontinuity}, parametric optimization problems \cite{morgan2004pseudocontinuity}, and the well-posedness of optimization problems \cite{anh2021levitin,anh2020well,anh2011well,morgan2006discontinuous}. On the other hand, \texttt{LQC} functions have been used to prove the existence of minimax solutions \cite{morgan2004new,scalzo2011existence}, the Tychonoff well-posedness of problems \cite{morgan2006discontinuous}, and to extend the Weierstrass theorem \cite{morgan2004new,amini2016some} and characterize the continuity of infima \cite{morgan2004new}. An extension of the Weierstrass extreme value theorem via the \texttt{WLC} assumption appears in  \cite{campbell1990optimization}.


Sequential versions of some of these  conditions have also been considered. Namely,  
sequentially lower quasi-continuous (\texttt{SLQC}) functions were introduced in \cite{morgan2004new}, and sequentially lower pseudo-continuous (\texttt{SLPC}) functions were introduced in \cite{morgan2004pseudocontinuity}, 
where \texttt{SLQC} has also appeared as sequentially lower weakly pseudo-continuous \cite[def 2.2]{morgan2006discontinuous}. These sequential versions share the same applications as \texttt{LQC} and \texttt{SLPC} functions since the sequential and non-sequential conditions are equivalent on first-countable spaces (cf. \cite[prop 2.3]{morgan2007pseudocontinuous} and \cite[prop 1]{scalzo2009uniform}) and in particular, on metric spaces.




\subsection{\texttt{LSC} from above}
An early version of the `\texttt{LSC} from above' notion is submonotonicity (\texttt{SM}), introduced in \cite{nemeth1986nonconvex} and referred to in \cite{bednarczuk2009vector}.
More recently, \cite{al2008some} introduced lower monotone (\texttt{LM}) functions, and \cite{chen2002note} and \cite{kirk2001b} introduced the notion of lower semi-continuous from above (\texttt{LSCA}).

The definitions of \texttt{LSCA} and \texttt{LM} can differ when the codomain is $\mathbb{R}$ depending on if limits to $\pm\infty$ are accepted (cf. \cite[rem 2.4]{al2008some}), although the two concepts are equivalent when the codmain is $\overline{\mathbb R}$.
Furthermore, the terminology for these two conditions vary in the literature, with the notions having been referred to as
sequentially lower semi-continuous from above \cite{aruffo2008generalizations,chen2014monotone}, sequentially lower monotonicity \cite{hamel2005equivalents, qiu2017vectorial}, monotonically semi-continuous \cite{bednarczuk2009vector}, partially lower semi-continuous \cite[ex 2.1.4]{borwein2005techniques}, and sequentially submonotone \cite{bao2022exact}.  

The \texttt{LSCA} assumption has been used in extensions of the Weierstrass extreme value theorem \cite{chen2014monotone, chen2002note, chen2018weak}, the Takahashi minimization theorem \cite{lin2006ekeland,lin2008maximal, park2024use, zhu2013caristi}, the Ekeland variation principle \cite{al2008some,chen2002note, kirk2001b, lin2006ekeland, qiu2020equilibrium, zhu2013caristi}, the Caristi fixed point theorem \cite{chen2002note,du2016generalized, kirk2001b, park2024use}, the Ky Fan minimax inequality \cite{chen2006note}, the von Neumann minimax theorem \cite{chen2018weak, lin2006ekeland, chen2006note}, and to characterize the existence of solutions of equilibrium problems \cite{castellani2010existence}. 

Formalizations of `\texttt{LSC} from above, near the infimum' are given by the below sequentially lower semi-continuous from above (\texttt{BLSCA}) and uniformly below sequentially lower semi-continuous from above (\texttt{UBLSCA}) conditions introduced by \cite[def 3.1]{bottaro2010some}. These conditions are weaker than \texttt{LSCA} and have been used in extensions of the Ekeland variational principle \cite{bottaro2010some, du2019some}, the Caristi fixed point theorem \cite{bottaro2010some, du2019some}, and the Takahashi minimization theorem \cite{du2019some}.

It is also possible to formalize the idea of `\texttt{LSC} strictly from above', such as via the
sequentially decreasing semi-continuity (\texttt{SDSC}) condition of \cite{gajek1994weierstrass}, also referred to as  
strictly decreasing lower semi-continuity by \cite{bao2022exact}. These functions have been used to extend the Weierstrass extreme value theorem \cite{gajek1994weierstrass} and the Ekeland variational principle \cite{bao2022exact}.

\subsection{Distillations of \texttt{LSC}}
We further partition the conditions that refine the notion of \texttt{LSC} to the \textit{strong} properties: those that allow approximation of the set of minimizers via the limits of minimizing nets, and the \textit{weak} properties: 
those that conclude non-emptiness and closure properties of the set of minimizers under some compactness condition.

\subsubsection{Strong properties}
Three main types of function classes allow for approximation of the set of minimizers: the regular global infimum (\texttt{RGI}) functions introduced in \cite{angrisani1982condizione} and appearing in the English literature in \cite{angrisani1996synthetic}; the inf-sequentially lower semi-continuous (\texttt{ISLSC}) functions, introduced in \cite{aruffo2008generalizations}; and the quasi-regular global infimum (\texttt{QRGI}) functions, introduced in \cite{amini2016some}.




\texttt{RGI} and \texttt{QRGI} functions have the desirable properties that they can be used to generate non-emptiness and relative compactness of the minimizer set \cite{amini2016some} and to generate well-posed optimization problems \cite{amini2016some}.
\texttt{RGI} functions have the additional desirable properties that they can characterize the existence of convergent minimizing nets \cite{amini2016some}, imply non-emptiness and compactness of the set of minimizers on non-compact domains \cite{kirk2000some}, and give conditions for the uniqueness of minimizers \cite{angrisani1996synthetic}.
Furthermore,
\texttt{RGI} functions have also been studied in the theory of fixed points \cite{ahmed2009some,angrisani1996synthetic,kirk2000some, gajic2014angrisani}.

\texttt{ISLSC} functions have been used in determining the Tychonoff well-posedness of optimization problems and in extensions of the Weierstrass extreme value theorem \cite{aruffo2008generalizations}.

\subsubsection{Weak properties}
Transfer continuous functions consist of the main class that makes conclusions regarding the non-emptiness of the minimizer set.
%
%
The main types of transfer continuous functions are the transfer weakly lower continuous (\texttt{TWLC}) functions, introduced in \cite{tian1995transfer}; the sequentially transfer weakly lower continuous (\texttt{STWLC}) functions, introduced in \cite{morgan2004new}; and the transfer lower continuous (\texttt{TLC}) functions, introduced in \cite{tian1995transfer}.
We refer the reader to \cite{zhou1992transfer} for the motivation behind transfer continuous functions.

\texttt{TWLC} functions precisely characterize the attainment of the minimum on compact domains (cf. \cite[lem 1]{amini2021fixed} and \cite[thm 1]{tian1995transfer}) and consequentially permit the statement of the most general extension of the Weierstrass extreme value theorem on compact sets.
Similarly, \texttt{STWLC} functions characterize the attainment of the minimum on sequentially compact domains \cite[thm 2.1]{morgan2004new}, while
\texttt{TLC} functions characterize the $\argmin$ set being closed and non-empty on compact domains \cite[thm 2]{tian1995transfer}. 

\texttt{TWLC} functions have been applied to establish of non-emptiness and relative compactness of the $\argmin$ set \cite{amini2021fixed,amini2016some}, as well as the existence of solutions to minimax problems \cite{scalzo2011existence}, and have found use primarily in the mathematical economics literature (see, e.g., \cite{nosratabadi2014partially, prokopovych2017strategic, scalzo2010pareto}).
\texttt{STWLC} functions share the applications of \texttt{TWLC} functions, since the two notions coincide on first countable spaces \cite[prop 2.1]{morgan2004new}. 

\texttt{TLC} functions have been used in extensions of the Berge maximum theorem \cite{tian1995transfer} and have applications to the theory of fixed points \cite{amini2021fixed} and bilevel optimization \cite{ding2006bilevel}. As with other transfer continuity assumptions, \texttt{TLC}  functions mainly appear in mathematical economics applications \cite{bueno2021existence, ding2003constrained, ding2003pareto, qiu4811574existence}.




\subsection{Continuity conditions that are not considered} 
Due to the scope of our survey, we resist covering continuity conditions that stray from the context of unstructured optimization problems on topological domains. However, we list some important examples of such conditions for the interested reader, including functions that are
decreasing semi-continuous towards a point, on metric spaces \cite{gajek1994weierstrass}, and functions that are upper semi-continuous from the right, on ordered spaces \cite{eisenfeld1975comparison, eisenfeld1975fixed}.
When the co-domain is a general ordered space,  order lower-semicontinuous functions \cite{deville2000vector}, lower quasi-continuous functions \cite{subiza1997numerical}, and transfer pseudo lower-continuous functions \cite[def 4]{zhou1992transfer} are applicable.

Several continuity conditions are specific to bi-functions, which include the generalized $t$-quasi-transfer
continuous  \cite[def 2]{scalzo2013essential}, quasi-transfer lower continuous  \cite[def 7]{tian1995transfer}, positively quasi-transfer
continuous  \cite[def 3]{scalzo2013essential}, $\gamma$-transfer lower continuous  \cite[def 3.2]{tian2017full}, transfer compactly lower semicontinuous  \cite[def 3.3]{chen2007coincidence}, 
and diagonally transfer lower continuous  \cite[def 4.1]{tian2017full} conditions.
Similarly, the lower 0-level closed \cite{anh2011well, anh2012well}, strongly 0-level closed \cite{anh2021levitin,anh2013well}, $b$-level lower semicontinuous \cite{anh2009well} and $b$-level quasi-lower semicontinuous \cite{anh2009well} conditions are exclusive to the context of equilibrium problems.



Lastly, we note that all continuity conditions that weaken \texttt{LSC} admit a complementary weakening of the upper semi-continuity (\texttt{USC}) version by requiring that $-f$ satisfies the particular condition.
Similarly, there are weaker notions of  continuity that require both $f$ and $-f$ to satisfy the particular weaker variant of \texttt{LSC}. 
Of particular note is the notion of pseudo-continuity, which has been extensively analyzed \cite{cotrina2022remarks, duy2022continuity, morgan2006discontinuous, morgan2007asymptotical,qiu2018approximation, scalzo2008pseudocontinuity,   scalzo2009hadamard}.
Such conditions fall outside of our scope since they replace the notion of continuity rather than \texttt{LSC}.

The remainder of the text proceeds as follows. In \cref{sec:Definitions}, we provide basic definitions of the reviewed continuity conditions, new continuity conditions, and some alternative formulations. In \cref{sec:Diagrams}, we present our main results regarding implications between the various continuity conditions via implication diagrams, and review results already known in the literature. Proofs of the veracity of the implication diagrams and various counterexamples are provided in \cref{sec:Proofs}. Pointers to known results regarding connections between the continuity conditions are collected in the Appendix.

\section{Definitions} \label{sec:Definitions}
Let $f:\mathcal X\to \overline{\mathbb R}$ with $\mathcal X$ a topological space and $\overline{\mathbb R}$ equipped with its standard topology  (see, e.g., \cite[p.g. 83]{oden2017applied}).
Unless otherwise specified $(x_n)$ and  $(x_\alpha)$ denote sequences and nets in $\mathcal X$, respectively. In this text, we  follow the terminology and definitions regarding nets from \cite{willard2012general}.

Using the shorthand $\inf_{\cal X}f=\inf_{x\in\cal X}f(x)$, we say that $(x_n)$ is a minimizing sequence if 
\[f(x_n)\to \inf_{\mathcal X}f,\]
with minimizing nets analogously defined. We denote the (potentially empty) set of minimizers by
\[\argmin(f) = \set{x\in\mathcal X:f(x) = \inf_{\mathcal X}f}.\] 

Writing the set of neighborhoods of $x\in \mathcal X$  as $\mathcal N(x)$, we use the notation
\begin{gather*}
\liminf_{y\to x}f(y) = \sup_{O\in\mathcal N(x)}\inf_{y\in O} f(y).
\end{gather*}
%
%
Similarly, for a net $(x_\alpha)$, we write 
\[\liminf_\alpha f(x_\alpha) = \lim_\alpha \inf_{\beta\geq \alpha} f(x_\beta)\] 
and say that $(y_\alpha)\subseteq \overline{\mathbb R}$
is decreasing if $\forall \alpha$ and $\forall \beta\geq \alpha$, $y_\beta\leq y_\alpha$. 
Further, we say that $(y_\alpha)$ is strictly decreasing if $\forall \alpha$ and $\forall \beta > \alpha$, $y_\beta < y_\alpha$.

The non-empty interval $(p,q)$ is a jump if $p,q\in f(\mathcal X)$ but there does not exist a $z\in\mathcal X$, such that 
\[p< f(z) < q,\] 
while the pair $(x,y)$ is a jump point if either $(f(x),f(y))$ or $(f(y),f(x))$ is a jump.
Further, we say there is a jump at $x$ if there exists a $y$, such that $(x,y)$ is a jump point. 

Note that the definition above does not always align with what one would consider a jump. 
For example, when $\cal X = \mathbb R$,
\[f(x) = \begin{cases}
    x & x\leq 0\\ 
    x+1 & x>0
\end{cases}\] 
has no jump, since $1\notin f(\cal X)$, while 
\[f(x) = \begin{cases}
    0 & x\leq 0\\
    1 & x>0
\end{cases},\] 
has jump $(0,1)$. 

\subsection{Continuity conditions}

We provide all definitions in a standardized form. Where our presentation differs to the original characterizations, we note that the definitions are equivalent. We also take the opportunity to introduce some novel continuity conditions that are missing from the literature.

Recall $\mathcal{X}$ is a topological space and $f:\mathcal{X}\to\overline{\mathbb{R}}$.
We say that $f$ is:
\begin{enumerate}
    \item \textbf{lower semi-continuous} (\texttt{LSC}) at $x\in \mathcal X$ iff
    \[f(x) \leq \liminf_{y\to x} f(y),\]
    \item \textbf{sequentially lower semi-continuous} (\texttt{SLSC}) at $x\in \mathcal X$  iff $\forall x_n\to x$, 
\[f(x)\leq \liminf_{n\to\infty} f(x_n),\] 
    \item \textbf{lower pseudo-continuous} (\texttt{LPC})  at $x\in \mathcal X$ iff $\forall y\in \mathcal X$ with $f(y) < f(x)$, 
\[f(y) < \liminf_{z\to x}f(z),\] 
    \item \textbf{sequentially lower pseudo-continuous} (\texttt{SLPC}) at $x\in\mathcal X$ iff $\forall y\in\mathcal X$ with $f(y) < f(x)$ and $\forall x_n\to x$, 
\[f(y) < \liminf_{n\to\infty} f(x_n),\] 
    \item \textbf{weakly lower continuous} (\texttt{WLC}) at $x\in\mathcal X$ iff $\forall y\in\mathcal X$ with $f(y) < f(x)$, $\exists U\in\mathcal N(x)$ such that 
    \[f(y) \leq \inf_U f,\] 
    \item \textbf{sequentially weakly lower continuous} (\texttt{SWLC}) at $x\in\mathcal X$ iff $\forall y\in\mathcal X$ with $f(y) < f(x)$, and $\forall x_n\to x$, for $n$ sufficiently large 
    \[f(y) \leq f(x_n),\] 
    \item \textbf{lower quasi-continuous} (\texttt{LQC}) at $x\in \mathcal X$  iff $\forall y\in \mathcal X$ with $f(y) < f(x)$,
    \[f(y)\leq \liminf_{z\to x}f(z),\] 
    \item \textbf{sequentially lower quasi-continuous} (\texttt{SLQC})  at $x\in \mathcal X$ iff $\forall y\in\mathcal X$ with $f(y) < f(x)$ and $\forall x_n\to x$, 
    \[f(y)\leq \liminf_{n\to\infty} f(x_n),\] 
    \item \textbf{partially lower continuous} (\texttt{PLC}) at $x$ iff 
    for each $y\in\mathcal X$ with $f(y)<f(x)$ such that $(x,y)$ is not a jump point, $\exists U\in\mathcal N(x)$ such that
    \[f(y) \leq \inf_U f,\] 
    \item \textbf{sequentially partially lower continuous} (\texttt{SPLC}) at $x\in\cal X$ iff 
    for each $y\in\mathcal X$ with $f(y)<f(x)$, such that $(x,y)$ is not jump point, and $\forall x_n\to x$, for $n$ sufficiently large 
    \[f(y) \leq f(x_n),\] 
    \item \textbf{submonotone} (\texttt{SM})  at $x\in \mathcal X$ iff $\forall x_\alpha \to x$ with $(f(x_\alpha))$ decreasing,
    \[f(x) \leq f(x_\alpha),\quad \forall \alpha,\] 
    \item \textbf{lower monotone} (\texttt{LM}) at $x\in \mathcal X$  iff $\forall x_n\to x\in\mathcal X$ with $(f(x_n))$ decreasing, 
    \[f(x) \leq f(x_n),\quad \forall n\in\mathbb N,\] 
    \item \textbf{lower semi-continuous from above} (\texttt{LSCA}) at $x\in \mathcal X$  iff $\forall x_n\to x$ with $(f(x_n))$ decreasing, 
    \[f(x)\leq \lim_{n\to\infty} f(x_n),\] 
    \item \textbf{decreasing semi-continuous} (\texttt{DSC})  at $x\in \mathcal X$ iff $\forall x_\alpha\to x$ with $(f(x_\alpha))$ strictly decreasing, 
    \[f(x)\leq \lim_\alpha f(x_\alpha), \] 
    \item \textbf{sequentially decreasing semi-continuous} (\texttt{SDSC}) at $x\in \mathcal X$  iff $\forall x_n\to x$ with $(f(x_n))$ strictly decreasing, 
    \[f(x) \leq \lim_{n\to\infty} f(x_n),\] 
    \item \textbf{regular global infimum} (\texttt{RGI}) at $x\in \mathcal X$ iff either $f(x)=\inf_{\mathcal X}f$ or $\exists U\in\mathcal N(x)$ such that 
\[\inf_{\mathcal X}f < \inf_U f,\]
    \item \textbf{inf-sequentially lower semi-continuous} (\texttt{ISLSC}) at $x\in \mathcal X$  iff $f(x)=\inf_{\mathcal X}f$ or there does not exist a minimizing sequence converging to $x$,
    \item \textbf{quasi-regular global infimum} (\texttt{QRGI}) at $x\in \mathcal X$ iff either
    \begin{enumerate}
      \item $\forall V\in\mathcal N(x)$, $\inf_{\mathcal X}f\in f(V)$,
      \item $\exists U\in \mathcal N(x)$ such that $\inf_{U}f > \inf_{\mathcal X}f$,
    \end{enumerate}
    \item \textbf{sequentially quasi-regular global infimum} (\texttt{SQRGI}) at $x\in \mathcal X$  iff 
    either there does not exist a minimizing sequence convergence to $x$ or there is a sequence $(x_n)\subset \argmin(f)$ with $x_n\to x$,
    \item \textbf{uniformly below lower semi-continuous from above} (\texttt{UBLSCA}) iff $\exists a\in(\inf_{\mathcal X}f,\, \infty]$ such that $\forall x\in \mathcal X$ and $\forall x_\alpha\to x$ with $(f(x_\alpha))$ decreasing and $f(x_\alpha)\leq a$,
\[f(x) \leq \lim_{\alpha} f(x_\alpha),\] 
    \item \textbf{uniformly below sequentially lower semi-continuous from above}\\ (\texttt{UBSLSCA}) iff $\exists a\in(\inf_{\mathcal X}f,\, \infty]$ such that $\forall x\in \mathcal X$ and $\forall x_n\to x$ with $(f(x_n))$ decreasing and $f(x_n)\leq a$,
\[f(x) \leq \lim_{n\to\infty} f(x_n),\] 
    \item \textbf{below lower semi-continuous from above} (\texttt{BLSCA}) at $x\in \mathcal X$  iff $\exists a\in(\inf_{\mathcal X}f,\, \infty]$ such that $\forall x_\alpha\to x$ with $(f(x_\alpha))$ decreasing and $f(x_\alpha)\leq a$, one has
\[f(x) \leq \lim_{\alpha} f(x_\alpha),\] 
    \item \textbf{below sequentially lower semi-continuous from above} (\texttt{BSLSCA}) at $x\in \mathcal X$ iff $\exists a\in(\inf_{\mathcal X}f,\, \infty]$ such that $\forall x_n\to x$ with $(f(x_n))$ decreasing and $f(x_n)\leq a$, one has
\[f(x) \leq \lim_{n\to\infty} f(x_n),\] 
    \item \textbf{transfer lower continuous} (\texttt{TLC}) at $x\in \mathcal X$ iff either $f(x) = \inf_{\mathcal X}f$ or $\exists y\in \mathcal X$ and $\exists U \in \mathcal N(x)$ such that $\forall z\in U$, 
\[f(y) < f(z),\] 
    \item \textbf{sequentially transfer lower continuous} (\texttt{STLC}) at $x\in \mathcal X$ iff  either $f(x) = \inf_{\mathcal X}f$ or $\exists y\in \mathcal X$ such that $\forall x_n\to x$ and for $n$ sufficiently large, 
    \[f(y) < f(x_n),\]     
    \item \textbf{transfer weakly lower continuous} (\texttt{TWLC}) at $x\in \mathcal X$  iff $\exists y\in \mathcal X$ and $\exists U\in \mathcal N(x)$, such that
\[f(y) \leq \inf_{U} f,\] 
    \item \textbf{sequentially transfer weakly lower continuous} (\texttt{STWLC})  at $x\in \mathcal X$  iff $\exists y\in \mathcal X$, such that $\forall x_n\to x$, 
\[f(y)\leq \liminf_{n\to\infty} f(x_n).\] 
\end{enumerate}
Where it applies, we say that $f$ itself satisfies the respective continuity condition if the condition holds $\forall x\in\cal X$.

\subsection{Equivalent definitions}
In the sequel, we provide some alternative characterizations of the continuity conditions above that we found useful, justify the name of the condition, or demonstrate agreement with prior definitions given in the literature.

\begin{proposition} (extension of \cite[lem 15]{amini2016some}) \label{lemma:twlc} 
    The following are equivalent: 
    \begin{enumerate}
        \item $f$ is $\mathtt{TWLC}$ at $x\in\mathcal X$,
        \item $\exists y\in\mathcal X$ such that 
        \[f(y) \leq \liminf_{z\to x} f(z),\] 
        \item  $\exists y\in \mathcal X$ such that for all nets $x_\alpha \to x$, 
        \[f(y) \leq \liminf_\alpha f(x_\alpha),\] 
        \item for all nets $x_\alpha\to x$ $\exists y\in\mathcal X$ such that 
        \[f(y) \leq \liminf_\alpha f(x_\alpha).\] 
    \end{enumerate}
\end{proposition}
\begin{proof}
    The implications $(1)\Rightarrow(2)$, $(2)\Rightarrow(3)$, and $(3)\Rightarrow(4)$ are trivial.\\ 
    $(4)\Rightarrow(1)$.
 If there is a neighbourhood
$U$ of $x$ such that $\inf_{U}f>\inf_{\mathcal{X}}f$, then (1)
is straightforward. So now assume that
\begin{gather*}
\inf_Uf=\inf_{\mathcal{X}}f,\quad\mbox{for all }U\in\mathcal{N}(x).
\end{gather*}
Consider the directed index set
\begin{gather*}
I=\left\{ (U,t): U\in\mathcal{N}(x),\; t>\inf_{\mathcal{X}}f \right\},
\end{gather*}
equipped with the direction $(U,t)\geq (U', t')$ when $U\subseteq U'$
and $t\leq t'$. 
If $f=\infty$, the result is trivial, so it may be assumed, without loss of generality, that $\inf_{\mathcal X}f<\infty$, whence $I\neq \emptyset$. 
For any index $\alpha = (U,t)$, take $x_\alpha \in U$ with $f(x_\alpha)\leq t$, noting that such an $x_\alpha$ exists.
Then $x_\alpha\to x$ with $f(x_\alpha)\to \inf_{\mathcal X}f$. 
By (4), $\exists y\in\mathcal X$ with 
\[f(y) \leq \liminf_\alpha f(x_\alpha) = \inf_{\mathcal X}f\] 
from which (1) is immediate.
\end{proof}
Replacing nets with sequences in item 3 justifies the definition of \texttt{STWLC}.

\begin{proposition} (\cite[prop 2.1]{morgan2004pseudocontinuity} and \cite[prop 2.1]{morgan2007pseudocontinuous})
    The following statements are equivalent:
    \begin{enumerate}
        \item $f$ is ($\mathtt{SLPC}$) $\mathtt{LPC}$,
        \item the set 
        \[\set{(x,y)\in \mathcal X\times f(\mathcal X): f(x)\leq y}\]
        is (sequentially closed) closed in $\mathcal X\times f(\mathcal X)$, 
        \item $\forall \lambda \in f(\mathcal X)$, 
        $$\lev_{\leq \lambda}:= \set{x\in\mathcal X:f(x) \leq \lambda}$$ 
        is (sequentially closed) closed in $\mathcal X$. 
    \end{enumerate}
\end{proposition}
In particular, item 3 is the primary definition used by \cite{quartieri2022existence}. 

We have the following equivalent definition for \texttt{QRGI}, which justifies the naming of \texttt{SQRGI}.

\begin{proposition}(extension of \cite[thm 6]{amini2016some})\label{lemma:qrgi}
    $f$ is $\mathtt{QRGI}$ at $x\in\cal X$ iff 
    there does not exist a minimizing net converging to $x$, or there is a net $(x_\alpha)\subset\mathrm{arg min}(f)$, with $x_\alpha \to x$. 
\end{proposition} 
\begin{proof}
    We will prove this result in two parts corresponding to the two cases in the definition of \texttt{QRGI}:
    \begin{enumerate}
        \item $\exists U\in\mathcal{N}(x)$, such that $\inf_Uf > \inf_{\mathcal X}f$ iff there does not exist a minimizing net converging to $x$, and
        \item $\forall V\in\mathcal N(x)$, $\inf_{\mathcal X}f \in f(V)$ iff there is a net $(x_\alpha) \subset \argmin(f)$ with $x_\alpha\to x$.
    \end{enumerate}
    $(1)$\\
    $(\Rightarrow)$ For any net, such that $x_\alpha \to x$, we have 
    \[\liminf_\alpha f(x_\alpha) \geq \inf_Uf >\inf_{\mathcal X}f,\] 
    so any such $(x_\alpha)$ cannot be minimizing.\\
    $(\Leftarrow)$ Proving the contrapositive, assume that $\forall U\in\mathcal N(x)$,
    \[\inf_Uf = \inf_{\mathcal{X}}f.\] 
    Then, as shown in the proof of \cref{lemma:twlc}, there is a net, such that $x_\alpha \to x$ with $f(x_\alpha)\to \inf_{\mathcal X}f$.
    That is, a convergent minimizing net exists. \\
    $(2)$\\ 
    $(\Rightarrow)$ Order $\mathcal N(x)$ by inclusion, and $\forall V\in \mathcal N(x)$, let $(x_V)\subset V$ be such that $f(x_V) = \inf_{\mathcal X}f$.
    Then, $(x_V)_{V\in\mathcal N(x)}$ is a net converging to $x$, with $(x_V) \subset \argmin(f)$.\\ 
    $(\Leftarrow)$ Take any $V\in\mathcal N(x)$. Then, for $\alpha$ large enough $x_\alpha \in V$, and thus $\inf_{\mathcal X}f\in f(V)$. 
\end{proof}

The following result establishes the analogous equivalent definition for \texttt{RGI}. Note that the result suggests that \texttt{ISLSC} could equivalently be referred to as sequential \texttt{RGI}.

\begin{proposition} (extension of \cite[thm 2]{amini2016some}) \label{lemma:rgi}
    $f$ is $\mathtt{RGI}$ at $x\in\cal X$ iff there does not exist a minimizing net converging to $x$, or $f(x)=\inf_{\mathcal X}f$.
\end{proposition}
\begin{proof}
    Analogous to the proof of \cref{lemma:qrgi}. 
\end{proof}

In \cite{nosratabadi2014partially} \texttt{PLC} is defined as a type of continuity for orderings.
This definition can be used to generate an ordering for functions $f:\mathcal X\to\overline{\mathbb R}$ by identifying the ordering $\succeq$ on $\mathcal X$ with 
\[x\succeq y\Longleftrightarrow f(x) \geq f(y).\] 
The following result demonstrates the equivalence of our definition of \texttt{PLC} to the direct translation of the ordering definition to functions.

\begin{proposition}
        $f$ is $\mathtt{PLC}$  iff $f$ has a countable number of jumps
        and 
    for each $x,y\in\mathcal X$ such that $f(y)<f(x)$, where $(x,y)$ is not a jump point, $\exists U\in\mathcal N(x)$ such that $\forall z\in U$, 
    \[f(y) \leq f(z).\] 
\end{proposition}
\begin{proof}
    It suffices to show that for any set $\mathcal X$ and any function $f:\mathcal X\to\overline{\mathbb R}$, $f$ has only countably many jumps. To obtain a contradiction, assume that $f$ has uncountably many jumps $(p_i,q_i)_{i\in I}$. 
    By definition $p_i<q_i$, and for any $i\neq j$, 
    \[(p_i,q_i)\cap (p_j,q_j)= \emptyset.\] 
    Writing $\aleph_0=|\mathbb N|$, if $\forall n\in\mathbb N$,
    \[|\set{i:(p_i,q_i)\subseteq (-n,n)}| \leq \aleph_0,\] 
    then, because a countable union of countable sets is also countable, we have
    \[|\set{i:(p_i,q_i)\subseteq (-\infty,\infty)}|\leq \aleph_0.\] 
    This implies the contradiction that $|I|\leq \aleph_0$, and thus it must instead be true that there is exists an $n^*\in\mathbb N$, such that 
    \[|\set{i:(p_i,q_i)\subseteq (-n^*,n^*)}| > \aleph_0.\] 

    Letting
    \[J = \set{i:(p_i,q_i)\subseteq (-n^*,n^*)},\] 
    we must then have 
    \[\sum_{i\in J} \left(q_i - p_i\right) \leq 2n^*.\] 
    Because $J$ is uncountable, by properties of uncountable sums, it must be true that for all but finitely many $j\in J$, $q_j-p_j=0$. By definition, $(p_j,q_j)$ is no longer a jump, yielding our result.
\end{proof}

\section{Implication diagrams} \label{sec:Diagrams}

Figures \ref{fig:main_point} to \ref{fig:known} constitute our main contribution. In particular, \Cref{fig:main_point} maps the implications between the various continuity conditions at some  $x\in\cal X$, while \cref{fig:main_everywhere} maps the implications when the continuity condition holds $\forall x\in\cal X$. We note that implications that cannot be obtained by traversing the diagrams do not hold in general.
To make clear our contribution, we map the implications and results that have already been reported in the literature in \Cref{fig:known}, with all other implications being novel (to the best of our knowledge). Specific references for each of the known implications and counterexamples are provided in the Appendix.

\begin{figure}[!]
    \centering
    \includegraphics[width=0.9\linewidth]{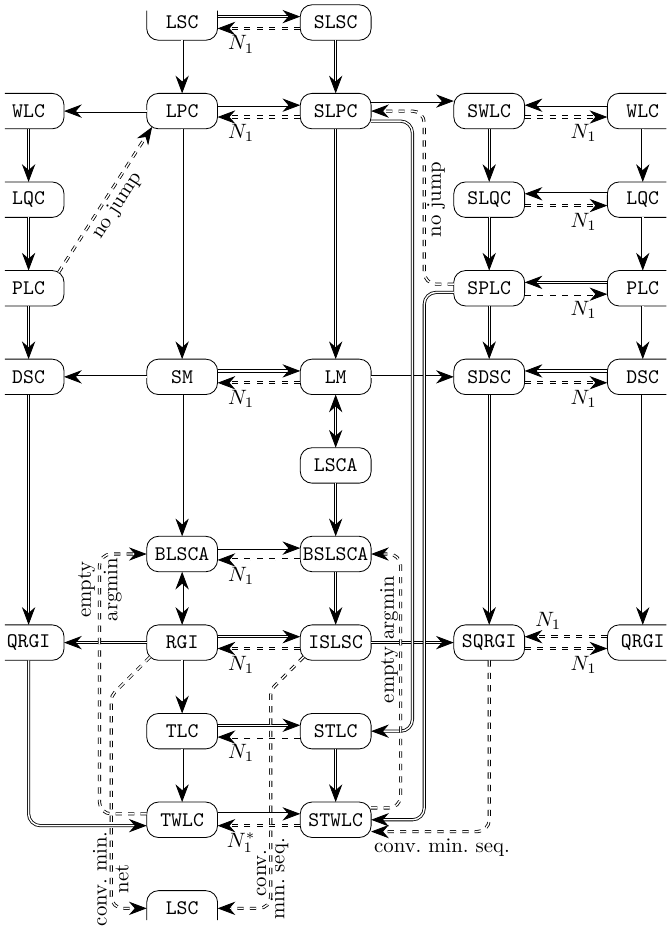}
    \caption{\textbf{Implications between continuity conditions which hold at a point $x\in\mathcal X$.} Solid lines indicate that the implication holds unconditionally, while `$N_1$' indicates that the implication holds when $\mathcal X$ is first countable. `Conv. min. seq/net.' indicates that the implication holds when a minimizing sequence/net converges to $x$, while `no jump' indicates that the implication holds when $f$ doesn't have a jump at $x$. `Empty argmin' means the implication holds if $\argmin(f) = \emptyset$. `$N_1^*$' indicates the implication holds either when $\mathcal X$ is first countable or there is a convergent minimizing sequence.}
    \label{fig:main_point}
\end{figure}

\begin{figure}[!]
    \centering
    \includegraphics[width=0.9\linewidth]{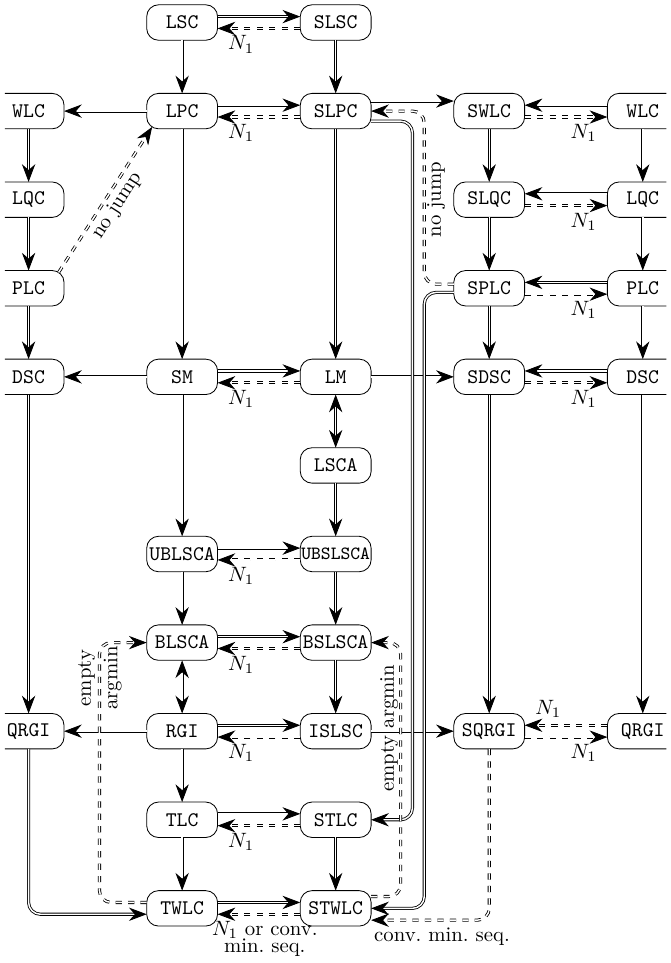}
    \caption{\textbf{Implications between continuity conditions which hold for $\forall x\in\cal X$.} Solid lines indicate that the implication holds unconditionally, while `$N_1$' indicates that the implication holds when $\mathcal X$ is first countable. `Conv. min. seq.' indicate that the implication holds when there is a converging minimizing sequence, while `no jump' indicates that the implication holds when $f$ has no jumps. `Empty argmin' means the implication holds if $\argmin(f) = \emptyset$.}
    \label{fig:main_everywhere}
\end{figure}

\begin{figure}[!]
    \centering
    \includegraphics[width=1\linewidth]{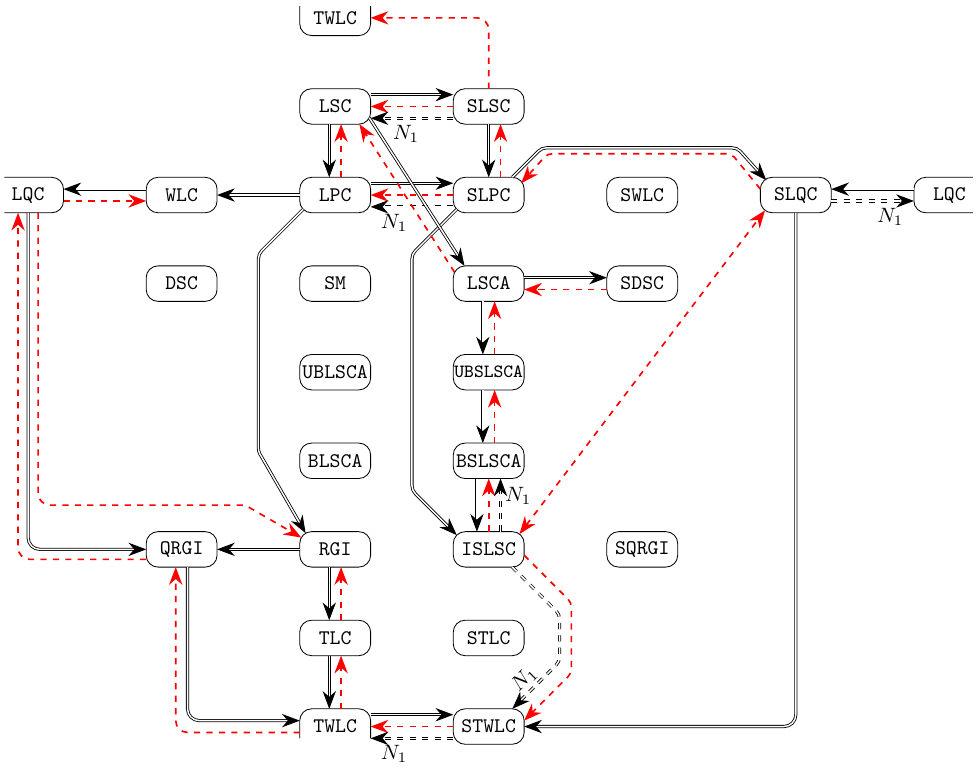}
    \caption{\textbf{Implications and results previously reported regarding continuity conditions that hold $\forall x\in\cal X$}. Black arrows indicate that the implication holds in general, while red arrows indicate that counterexamples exist that prevent the corresponding implication. `$N_1$' indicates that the implication holds when $\mathcal X$ is first countable.}
    \label{fig:known}
\end{figure}

\section{Proofs} \label{sec:Proofs}

\subsection{Implications between continuity conditions which hold at a point} 
In this section, we prove that the diagram in \cref{fig:main_point} is correct. 
 To begin, note that the following implications are trivial by definition: 

\[
\begin{array}{lll}
(\mathtt{LSC})\Rightarrow(\mathtt{LPC}) & (\mathtt{SLSC})\Rightarrow(\mathtt{SLPC}) & (\mathtt{LPC})\Rightarrow(\mathtt{WLC}) \\
(\mathtt{WLC})\Rightarrow (\mathtt{LQC}) & (\mathtt{WLC})\Rightarrow (\mathtt{SWLC}) & (\mathtt{SLPC})\Rightarrow(\mathtt{SWLC}) \\
(\mathtt{SWLC})\Rightarrow(\mathtt{SLQC}) & (\mathtt{DSC})\Rightarrow(\mathtt{SDSC}) & (\mathtt{SM})\Rightarrow(\mathtt{DSC}) \\
(\mathtt{SM})\Rightarrow(\mathtt{LM}) & (\mathtt{SM})\Rightarrow (\mathtt{BLSCA}) & (\mathtt{LM})\Leftrightarrow(\mathtt{LSCA}) \\
(\mathtt{LSCA})\Rightarrow(\mathtt{BSLSCA}) & (\mathtt{LM})\Rightarrow(\mathtt{SDSC}) & (\mathtt{BLSCA})\Rightarrow(\mathtt{BSLSCA}) \\
(\mathtt{RGI})\Rightarrow(\mathtt{QRGI}) & (\mathtt{RGI})\Rightarrow(\mathtt{TLC}) & (\mathtt{ISLSC})\Rightarrow(\mathtt{SQRGI}) \\
(\mathtt{TLC})\Rightarrow(\mathtt{TWLC}) & (\mathtt{STLC})\Rightarrow(\mathtt{STWLC}) & (\mathtt{TWLC})\Rightarrow(\mathtt{STWLC}) \\
(\mathtt{PLC})\Rightarrow(\mathtt{SPLC}) & (\mathtt{LQC})\Rightarrow(\mathtt{SLQC}) & (\mathtt{LPC})\Rightarrow(\mathtt{SLPC}) \\
(\mathtt{TLC})\Rightarrow(\mathtt{STLC}) & & \\
\end{array}
\]

$(\mathtt{RGI})\Rightarrow(\mathtt{ISLSC})$ is trivially proved by  \cref{lemma:rgi}.
$(\mathtt{LCS})\Rightarrow(\mathtt{SLSC})$ and $(\mathtt{SLSC}+N_1)\Rightarrow(\mathtt{LSC})$ are well known; see, e.g. \cite[rem 1.3.16]{denkowski2013introduction}.
$(\mathtt{SLQC}+N_1)\Rightarrow(\mathtt{LQC})$ is proved in \cite[prop 1]{scalzo2009uniform}, and $(\mathtt{STWLC}+N_1)\Rightarrow(\mathtt{TWLC})$ is proved in \cite[prop 2.1]{morgan2004new}. \\

\noindent $(\mathtt{SLPC}+N_1)\Rightarrow(\mathtt{LPC})$\\
This implication is proved in  \cite[prop 2.3]{morgan2007pseudocontinuous}, although we take the opportunity to present an alternative proof.
Let $(V_n)_{n\in\mathbb N}$ be a neighbourhood basis for $x$.
From each $V_n$, pick $x_n\in V_n$ such that
\[\inf_{m\geq n} f(x_m)\leq f(x_n)\leq \inf_{V_n} f+\frac1n\text.\]
By definition of \texttt{SLPC}, for any $y\in\mathcal X$ with $f(y) < f(x)$, it holds that
\begin{gather*}
f(y)<\liminf_{n\to\infty} f(x_n)
=\lim_{n\to\infty}\inf_{m\geq n} f(x_m)
\leq \lim_{n\to\infty}\left( \inf_{V_n}f + \frac1n \right)\\
\leq \sup_n\inf_{V_n}f + \lim_{n\to\infty}\frac1n
=\sup_{V\in\mathcal{N}(x)} \inf_V f=\liminf_{z\to x} f(z),
\end{gather*}
as required.

\vspace{\baselineskip}

\noindent $(\mathtt{SWLC}+N_1)\Rightarrow (\mathtt{WLC})$\\ 
    We seek to prove the contrapositive by assuming that $\mathcal X$ is $N_1$ but $f$ is not \texttt{WLC} at $x\in\cal X$. 
    By definition, $\exists y\in\mathcal X$ with $f(y) < f(x)$, such that $\forall U\in\mathcal N(x)$, it holds that 
    \[\inf_Uf < f(y)\text.\] 
    Let $(V_n)_{n\in\mathbb N}$ be a neighbourhood base of $x$. 
    Select $x_n\in V_n$ with $f(x_n) < f(y)$.
    Then, $x_n\to x$, but $\forall n\in \mathbb N$, 
    \[f(x_n) < f(y).\] 
    This implies that $f$ is not \texttt{SWLC}, as required. 

    \vspace{\baselineskip}

    \noindent   $(\mathtt{SPLC}+N_1)\Rightarrow(\mathtt{PLC})$ follows via an analogous argument.

    \vspace{\baselineskip}
    
    \noindent $(\mathtt{LPC})\Rightarrow(\mathtt{SM})$\\ 
    To obtain a contradiction, assume that $f$ is \texttt{LPC} but not \texttt{SM}.
    Because $f$ is not \texttt{SM}, $\exists x_\alpha\to x$, such that $(f(x_\alpha))$ is decreasing and $\exists \beta$, such that
    \[f(x) > f(x_\beta)\text.\]
    Because $(f(x_\alpha))$ is decreasing, this means that $\forall \alpha \geq \beta$, it holds that
    \[f(x) > f(x_\alpha).\] 
    By \texttt{LPC}, this implies that 
    \[\liminf_{\beta} f(x_\beta) > f(x_\alpha),\qquad \forall \alpha\geq \beta.\] 
    But this is impossible since $(f(x_\alpha))$ is decreasing, yielding the required contradiction. 
    
    \vspace{\baselineskip}
    
    \noindent $(\mathtt{SLPC})\Rightarrow(\mathtt{LM})$ follows via an analogous argument.

    \vspace{\baselineskip}

    \noindent $(\mathtt{SPLC})\Rightarrow(\mathtt{STWLC})$\\ 
    We must consider the two cases:
    \begin{enumerate}
        \item $\exists y\in\mathcal X$ with $f(y)<f(x)$, such that $(x,y)$ is not a jump point, and
        \item $\forall y\in\mathcal X$ with $f(y)<f(x)$, $(x,y)$ is a jump point.
    \end{enumerate}
    The result is trivial for case 1. 
    For case 2, it must be true that 
    \[\left|\set{f(y):y\in\mathcal X,\,f(y)<f(x)}\right|\in \set{0,1},\] 
    or else there would be some $y\in\cal X$, such that $(x,y)$ is not a jump point. 
    This implies that $f$ attains its minima and is therefore \texttt{STWLC}, as required.

    \vspace{\baselineskip}


    \noindent Before continuing, we require the following result.
    We say that $(x_\alpha)_{\alpha\in A}$ is non-trivial, if $\forall \alpha \in A$, $\exists \beta\in A$ with $\beta > \alpha$. That is, a net is trivial if $A$ contains a maximal element.
    Further, we say that $(x_\alpha)$ is eventually constant if $\exists \beta$, such that $\forall \alpha \geq \beta$, $x_\alpha = x_\beta$. 
    \begin{lemma} \label{lemma:decsub}
        Let $(x_\alpha)_{\alpha\in I}\subseteq \mathbb R$ be a net, with $x_\alpha \to \inf_\alpha x_\alpha$. Then, either $(x_\alpha)$ is eventually constantly equal to  $\inf_\alpha x_\alpha$, or $(x_\alpha)$ has a non-trivial, strictly decreasing subnet. 
    \end{lemma}

    \begin{proof}
        When $(x_\alpha)$ is eventually constant (which includes the special case where the net $(x_\alpha)$ is trivial), the result is immediate.
        We thus restrict our analysis to the case when $(x_\alpha)$ is non-trivial and not eventually constant. 
        Because any subnet of a non-trivial net is also non-trivial (otherwise, the cofinal property will be violated), the result is proved if we can show that there is a strictly decreasing subnet. 

        Let 
        \[J = \set{(\alpha,x_\alpha):\alpha\in I,\,x_\alpha > \inf_\beta x_\beta},\] 
        and define the order $\succ$ on $J$ by
        \[(\alpha,x_\alpha)\succ (\alpha',x_{\alpha'})\Longleftrightarrow \alpha>\alpha',\, x_\alpha < x_{\alpha'}.\]
        Further, define the order $\succeq$ on $J$ by 
        \[(\alpha,x_\alpha)\succeq (\alpha',x_{\alpha'})\Longleftrightarrow \alpha = \alpha' \text{ or } (\alpha,x_\alpha)\succ (\alpha',x_{\alpha'}).\] 
        We note that $\succeq$ is a pre-order and will show that $\succeq$ is directed. 
        Take $(\alpha,x_\alpha),(\beta,x_\beta)\in J$. 
        Then, because $x_\alpha \to \inf_\alpha x_\alpha$, $\exists \gamma$ such that $\forall \lambda \geq \gamma$, 
        \[x_\lambda < \min\set{x_\alpha,x_\beta}.\] 
        Observe that there must exist $\lambda \geq \alpha,\beta,\gamma$, such that $x_\lambda > \inf_\alpha x_\alpha$, or else $(x_\alpha)$ would eventually be constant. 
        Taking any such $\lambda$, we obtain 
        \[(\lambda,x_\lambda) \succeq (\alpha,x_\alpha),(\beta,x_\beta).\] 
        
        Now let $\phi:J\to I$ be given by 
        \[\phi(\alpha,x_\alpha)=\alpha.\] 
        By definition of $\succeq$, this map is monotone increasing. 
        Because $(x_\alpha)$ is not eventually constant, $\forall \alpha \in I$, $\exists \beta\geq \alpha$, such that $x_\beta \neq \inf_\alpha x_\alpha$. 
        This implies that $\phi$ is also cofinal.
        Finally, by definition of $\succeq$, $x\circ \phi$ is strictly decreasing, as required.
    \end{proof}

    \noindent $(\mathtt{DSC})\Rightarrow(\mathtt{QRGI})$\\ 
    If there is no minimizing net converging to $x$, then  $f$ is \texttt{QRGI} by \cref{lemma:qrgi}. 
    Thus, assume there is a minimizing net $(x_\alpha)$ converging to $x$,
    if $(f(x_\alpha))$ is eventually constant, then the result follows by \cref{lemma:qrgi}.
    If not, then by \cref{lemma:decsub}, there is a subnet $(x_{\alpha_\beta})$ such that $(f(x_{\alpha_\beta}))$ is strictly decreasing.
    Because $f$ is \texttt{DSC},
     \[f(x)\leq \lim_\beta f(x_{\alpha_\beta}) = \inf_{\mathcal X}f.\] 
    That is, $f(x) = \inf_{\mathcal X}f$ and thus $f$ is trivially \texttt{LSC} at $x$. In particular, it follows that $f$ is \texttt{QRGI} at $x$. 
    
    \vspace{\baselineskip}

    \noindent $(\mathtt{SDSC})\Rightarrow(\mathtt{SQRGI})$ follows via an analogous argument.
    
    \vspace{\baselineskip}

    \noindent For the next set of implications, we require the following result.
    \begin{lemma} \label{lemma:subnet}
        Let $\mathcal X$ be $N_1$, $f:\mathcal X\to \overline{\mathbb R}$, and  $(x_\alpha)_{\alpha\in I}\subseteq \mathcal X$. If $x_\alpha \to x$ and $f(x_\alpha)\to c$ then there is an increasing map $\phi:I\to \mathbb N$ such $x(\phi(n))\to x$ and $f(x(\phi(n)))\to c$. 
        If $(x_\alpha)$ is non-trivial, we can additionally take $\phi$ strictly increasing. 
    \end{lemma}
    \begin{proof}
        Let $(V_n)$ be a neighbourhood base of $x$ and let $(U_n)$ be a neighbourhood base of $c$. 
        Then $\forall n\in\mathbb N$, $\exists \beta_n\in I$ such that $\forall \alpha \geq \beta_n$, $x\in V_n$.
        Similarly, $\forall n\in\mathbb N$ $\exists \gamma_n\in I$ such that $\forall \alpha \geq \gamma_n$, $f(x_\alpha)\in U_n$. 
        Pick $\phi(1)$ arbitrarily and recursively take $\phi(n) \geq \beta_n,\gamma_n, \phi(n-1)$. 
        $\phi$ is then increasing, $x(\phi(n))\to x$ and $f(x(\phi(n)))\to c$.

        If $(x_\alpha)$ is non-trivial we can instead take $\phi(n)> \beta_n,\gamma,\phi(n-1)$ to make $\phi$ strictly increasing.
    \end{proof}
    
    \vspace{\baselineskip}
    
    \noindent $(\mathtt{SDSC}+N_1)\Rightarrow(\mathtt{DSC})$\\
    Take any $x_\alpha\to x$ with $(f(x_\alpha))$ strictly decreasing.
    If $(x_\alpha)$ is trivial, then 
    \[\lim_\alpha f(x_\alpha) = f(x)\] 
    and the definition of \texttt{DSC} is satisfied.
    If $(x_\alpha)$ is non-trivial, \cref{lemma:subnet} implies that there is a sequence $(x_n)$ with $\set{x_n}\subseteq \set{x_\alpha}$, where $(f(x_n))$ is strictly decreasing, $x_{n}\to x$ and $\lim_\alpha f(x_\alpha) = \lim_{n}  f(x_n)$.
    Because $f$ is \texttt{SDSC},
    \[f(x)\leq \lim_{n\to\infty} f(x_{\alpha_n}) = \lim_\alpha f(x_\alpha).\]
    And because $(x_\alpha)$ was arbitrarily chosen, we conclude that $f$ is \texttt{SM}, as required.

    \vspace{\baselineskip}

    \noindent $(\mathtt{LM}+N_1)\Rightarrow(\mathtt{SM})$ and $(\mathtt{BSLSCA}+N_1)\Rightarrow(\mathtt{BLSCA})$ follow via analogous arguments. 
    
    \vspace{\baselineskip}

    \noindent $(\mathtt{QRGI})\Rightarrow(\mathtt{TWLC})$\\ 
    From the definition of \texttt{QRGI}, either 
    \begin{enumerate}
        \item $\forall V\in\mathcal N(x)$, $\inf_{\mathcal X}f \in f(V)$, or
        \item $\exists U\in\mathcal N(x)$, such that $\inf_{U}f > \inf_{\mathcal X}f$.
    \end{enumerate}
    (1) implies that the infimum is attained; thus, $f$ is trivially \texttt{TWLC}. 
    If (2) is true, then \texttt{TWLC} is verified by the existence of $y\in\mathcal{X}$ such that $\inf_{\mathcal{X}} f < f(y) < \inf_U f$.

    \vspace{\baselineskip}

    \noindent $(\mathtt{SQRGI}+N_1)\Rightarrow(\mathtt{QRGI})$\\ 
    Because $f$ is \texttt{SQRGI}, either 
    \begin{enumerate}
        \item there is no minimizing sequence converging to $x$, or
        \item $\exists (x_n)\subset\argmin(f)$, such that $x_n\to x$.
    \end{enumerate}
We proceed by showing that
\begin{itemize}
\item $N_1$ + (1) $\implies$ \texttt{QRGI} condition (b); i.e., $\exists U\in\mathcal{N}(x)$
such that $\inf_U f>\inf_{\mathcal{X}}f$;
\item (2) $\implies$ \texttt{QRGI} condition (a); i.e., $\forall V\in \mathcal{N}(x)$, $\inf_{\mathcal{X}} f\in f(V)$.
\end{itemize}
The second bullet is trivial and we do not need the $N1$ assumption here, since
the sequence $(x_n)$ would eventually enter $V$ by the definition of convergence.

For the first bullet, we give the proof by contradiction.
Indeed, suppose that $\inf_U f=\inf_{\mathcal{X}} f$, $\forall U\in\mathcal{N}(x)$.
Because $\mathcal X$ is $N_1$, there is a countable neighborhood basis $(U_n)$ of $x$.
For each $n\geq 1$, take $y_n\in U_n$ such that $f(y_n)<\inf_{U_n} + \frac1n=\inf_{\mathcal{X}}f +\frac1n$.
This gives $y_n\to x$ and $f(y_n)\to \inf_{\mathcal{X}}f$, contradicting the above item
(1) of \texttt{SQRGI}.

    \vspace{\baselineskip}

    \noindent $(\mathtt{QRGI}+N_1)\Rightarrow (\mathtt{SQRGI})$\\ 
    By \cref{lemma:qrgi}, either 
    \begin{enumerate}
        \item there is no minimizing net converging to $x$, or
        \item there is a net $(x_\alpha) \subset \argmin_{\mathcal X}f$ with $x_\alpha\to x$.
    \end{enumerate}
    If (1) is true, no minimizing sequence converges to $x$. Thus, $f$ is \texttt{SQRGI}. 
    If (2) is true, then by \cref{lemma:subnet}, there is a sequence $(x_n)\subset\argmin_{\mathcal X}f$ with $x_n\to x$, and the conclusion follows. 

    \vspace{\baselineskip}

    \noindent $(\mathtt{SQRGI}+\mathrm{conv.}\,\mathrm{min.}\,\mathrm{seq})\Rightarrow(\mathtt{STWLC})$\\ 
    Because there is a minimizing sequence converging to $x$, \texttt{SQRGI} implies that the infimum is attained, and thus
    $f$ is trivially \texttt{STWLC}.  

    \vspace{\baselineskip}

    \noindent $(\mathtt{BLSCA})\Rightarrow(\mathtt{RGI})$\\ 
    If there is no minimizing net converging to $x$, then $f$ is \texttt{RGI} by \cref{lemma:rgi}.
    Thus, it suffices to assume that a minimizing net $(x_\alpha)$ converges to $x$.
    By passing to a subnet, \cref{lemma:decsub} implies that we can assume, without loss of generality, that $(f(x_\alpha))$ is decreasing and thus for any $a>\inf_{\mathcal X}f$, $f(x_\alpha) \leq a$ for $\alpha$ sufficiently large.
    Then, by \texttt{BLSCA}, 
    \[\lim_\alpha f(x_\alpha) \geq f(x)\Longrightarrow f(x) = \inf_{\mathcal X}f,\] 
    and we may conclude that $f$ is \texttt{RGI}.

    \vspace{\baselineskip}
    \noindent $(\mathtt{BSLSCA})\Rightarrow(\mathtt{ISLSC})$ via an analogous argument.

    \vspace{\baselineskip}

    \noindent $(\mathtt{RGI})\Rightarrow(\mathtt{BLSCA})$\\ 
    If $f(x) = \inf_{\mathcal X}f$, then $f$ is trivially \texttt{BLSCA} at $x$, and 
    if $f(x) > \inf_{\mathcal X}f$, then by \texttt{RGI}, $\exists U\in\mathcal N(x)$ such that 
    \[\inf_{U}f > \inf_{\mathcal X}f.\] 
    But for any $a\in (\inf_{\mathcal X}f, \inf_Uf)$,
    there does not exist a net $(x_\alpha)$ with $x_\alpha \to x$, such that $f(x_\alpha)\leq a$.
    Therefore, the definition of \texttt{BLSCA} is vacuously satisfied.

     \vspace{\baselineskip}
    
    \noindent $(\mathtt{ISLSC}+N_1)\Rightarrow(\mathtt{RGI})$\\ 
    %
    If $f(x) = \inf_{\mathcal X}f$, then we have the required conclusion. 
    By \cref{lemma:rgi}, it, therefore, suffices to show there is no convergent minimizing net. 
    If there was such a net, then \cref{lemma:subnet} implies a minimizing sequence converging to $x$.
    But this contradicts the definition of  \texttt{ISLSC}, and the implication follows.

    \vspace{\baselineskip}

    \noindent $(\mathtt{STLC}+N_1)\Rightarrow(\mathtt{TLC})$\\ 
    To prove the contrapositive, we assume that $\mathcal X$ is $N_1$ but $f$ is not \texttt{TLC}. 
    This implies that $f(x) > \inf_{\mathcal X}f$ and $\forall U\in \mathcal N(x)$, $\forall y\in \mathcal X$, $\exists z\in U$ with $f(z) \leq f(y)$. 
    Consider the set 
    \[J = \set{(U,y): U\in\mathcal N(x),\, y\in f(\mathcal X)},\] 
    with the order
    \[(U,y)\succeq (U',y')\Leftrightarrow U\subseteq U',\, y\leq y',\]    
    and note that this ordering is directed.
    Then, $\forall \alpha = (U,y)\in J$, take $x_\alpha \in U$ with $f(x_\alpha) \leq y$. 
    We have 
    \[x_\alpha \to x\qquad f(x_\alpha) \to \inf_{\mathcal X}f,\] 
    and by \cref{lemma:subnet}, there is a sequence $x_n\to x$ with $f(x_n)\to \inf_{\mathcal X}f$. 
    Because $(x_n)$ is minimizing, we have
    \[f(x_n) \leq y,\] 
    $\forall y\in f(\mathcal X)$ and $n$ sufficiently large.
    This implies that $f$ is not \texttt{STLC}, as required.\\

    \noindent $(\mathtt{PLC})\Rightarrow(\mathtt{DSC})$\\ 
    Let $(x_\alpha)$ be a net converging to $x$ with $(f(x_\alpha))$ strictly decreasing. 
    If $(x_\alpha)$ is trivial, then we trivially have
    \[f(x) = \lim_\alpha f(x_\alpha).\] 
    It thus suffices to assume that $(x_\alpha)$ is non-trivial.
    To obtain a contradiction, assume that
    \[\lim_\alpha f(x_\alpha) < f(x).\] 
    Then $\exists \gamma$, such that $\forall \alpha \geq \gamma$,
    \[f(x_\alpha) < f(x).\] 
    For any $\beta > \gamma$, because $(f(x_\alpha))$ is strictly decreasing, we have that $(x, x_\beta)$ is not a jump point.
    By definition of \texttt{PLC}, $\exists U\in\mathcal N(x)$ such that, 
    \[f(x_\beta) \leq \inf_U f.\] 
    But, because $(x_\alpha)$ is non-trivial and converges to $x$, $\exists \gamma> \beta$, such that $x_\gamma \in U$.
    We therefore have 
    \[f(x_\gamma) < f(x_\beta) \leq f(x_\gamma),\] 
    and thus our required contradiction.
    
    \vspace{\baselineskip}

    \noindent $(\mathtt{SPLC})\Rightarrow(\mathtt{SDSC})$ follows via an analogous argument.
    
    \vspace{\baselineskip}

    \noindent $(\mathtt{LQC})\Rightarrow(\mathtt{PLC})$\\ 
    Assume there is some $y\in\mathcal{X}$ such that $f(y)<f(x)$ and $(x,y)$ is not a jump point, else the result is trivial. 
    Then there exists some $y'\in\mathcal{X}$ such that
    \begin{gather*}
    f(y)<f(y')<f(x).
    \end{gather*}
    From \texttt{LQC},
    \begin{gather*}
    f(y)<f(y')\leq\liminf_{z\to x}f(z).
    \end{gather*}
    So there exists some $U\in \mathcal{N}(x)$ such that
    \begin{gather*}
    f(y)\leq \inf_U f,
    \end{gather*}
    which completes the proof.
    
    \vspace{\baselineskip}
    
    \noindent $(\mathtt{SLQC})\Rightarrow(\mathtt{SPLC})$ follows via an analogous argument.

    \vspace{\baselineskip}

    \noindent 
    $(\mathtt{PLC}+\text{no. jump})\Rightarrow(\mathtt{LPC})$\\ 
    Take any $y\in\mathcal X$ with $f(y)<f(x)$. 
    Because $f$ has no jumps at $x$, $\exists y'\in\mathcal X$ such that 
    \[f(y) < f(y') < f(x).\] 
    Because $f$ is \texttt{PLC}, $\exists U\in\mathcal N(x)$ such that 
    \[f(y') \leq \inf_Uf.\] 
    In particular 
    \[f(y) < f(y') \leq \liminf_{z\to x} f(z)\] 
    and thus $f$ is \texttt{LPC}, as required. 

    \vspace{\baselineskip}

    \noindent
    $(\mathtt{SPLC}+\text{no jump})\Rightarrow(\mathtt{SLPC})$ follows via an analogous argument.
    
    \vspace{\baselineskip}

    \noindent 
    $(\mathtt{RGI}+\text{conv. min. net})\Rightarrow(\mathtt{LSC})$\\ 
    Because $f$ is \texttt{RGI}, either 
    \begin{enumerate}
        \item $f(x)=\inf_{\mathcal X}f$, or
        \item $\exists U\in\mathcal N(x)$ such that $\inf_{U}f > \inf_{\mathcal X}f$.
    \end{enumerate}
    Because there is a minimizing net converging to $x$, case (2) cannot be true. 
    It then follows that $f(x)=\inf_{\mathcal X}f$, and therefore $f$ is trivially \texttt{LSC} at $x$.

    \vspace{\baselineskip}

    \noindent
    $(\mathtt{ISLSC}+\text{conv. min. seq.})\Rightarrow(\mathtt{LSC})$\\ 
    Because $f$ is \texttt{ISLSC}, either 
    \begin{enumerate}
        \item $f(x)=\inf_{\mathcal X}f$, or
        \item there is no minimizing sequence converging to $x$.
    \end{enumerate}
    By assumption, case (1) is true; thus, $f$ is trivially \texttt{LSC} at $x$.

    %
    \vspace{\baselineskip}

    \noindent
    $(\mathtt{TWLC}+\text{empty }\argmin)\Rightarrow(\mathtt{BLSCA})$\\ 
    This implication appears in \cite[fig 1]{amini2016some}, although we take the opportunity to provide an alternative proof.
    Because $\argmin(f)=\emptyset$, \texttt{TWLC} implies that $\exists y\in\mathcal X$ and $\exists U\in\mathcal N(x)$, such that 
    \[\inf_{\mathcal X}f < f(y) \leq \inf_{U}f.\] 
    The definition of
    \texttt{BLSCA} is then vacuously satisfied
    by taking any $a\in (\inf_{\mathcal X}f, f(y))$. 

    \vspace{\baselineskip}

    \noindent
    $(\mathtt{STWLC}+\text{empty }\argmin)\Rightarrow(\mathtt{BSLSCA})$ follows via an analogous argument.

    \vspace{\baselineskip}

    \noindent $(\mathtt{STWLC}+\text{conv. min. seq})\Rightarrow(\mathtt{TWLC})$\\ 
    Let $(x_n)$ be a minimizing sequence converging to $x$. Then by \texttt{STWLC} there is some $y\in\mathcal X$ such that 
    \[\inf_{\mathcal X}f \leq f(y) \leq \liminf_{n\to\infty} f(x_n) = \inf_{\mathcal X}f\text.\] 
    That is, the infimum of $f$ is attained at $y$, and so $f$ is \texttt{TWLC}.

    \vspace{\baselineskip}

    \noindent $(\mathtt{SLPC})\Rightarrow(\mathtt{STLC})$\\ 
    If $f(x)=\inf_\mathcal Xf$, then $f$ is trivially \texttt{STLC}. 
    On the other hand, if $f(x) > \inf_{\mathcal X}f$, then there exists a $y\in\mathcal X$ such that $f(y) < f(x)$. 
    By definition of \texttt{SLPC}, for any sequence $(x_n)$ such that $x_n\to x$, it holds that 
    \[f(y) < \liminf_{n\to\infty} f(x_n).\] 
    In particular, for sufficiently large $n$, $f(x) < f(x_n)$, implying that $f$ is \texttt{STLC}.

\subsection{Implications between continuity conditions over the entire domain}
This section argues that the diagram in \cref{fig:main_everywhere} is correct. 
Most implications are obtained by applying the results in \cref{fig:main_point} at every $x\in\mathcal X$. 
Implications involving \texttt{UBLSCA} and \texttt{UBSLSCA} do not appear in  \cref{fig:main_point} as these are not pointwise concepts. Among these new implications, the following are trivial:
\[\begin{array}{lll}
     (\mathtt{SM})\Rightarrow(\mathtt{UBLSCA}) & (\mathtt{LM})\Rightarrow(\mathtt{UBSLSCA}) & (\mathtt{UBLSCA})\Rightarrow(\mathtt{UBSLSCA})\\
     (\mathtt{UBLSCA})\Rightarrow(\mathtt{BLSCA})& (\mathtt{UBSLSCA})\Rightarrow(\mathtt{BSLSCA}) &
\end{array}\]
$(\mathtt{UBSLSCA}+N_1)\Rightarrow(\mathtt{UBLSCA})$ follows via an analogous argument to that for the implication $(\mathtt{LM}+N_1)\Rightarrow(\mathtt{SM})$. The implications
$(\mathtt{SQRGI}+\text{conv. min. seq})\Rightarrow(\mathtt{STWLC})$ and $(\mathtt{STWLC}+\text{conv. min. seq})\Rightarrow(\mathtt{TWLC})$ over the entire domain do not follow immediately from the corresponding pointwise implications, however the proofs for these results follow analogous arguments.

\subsection{Counter examples}
In this section, we prove that no further implications can be inferred other than those obtained by traversing the diagrams from \cref{fig:main_point} and \cref{fig:main_everywhere}. 
To this end, we show that the arrows in \cref{fig:counter_examples} each have counterexamples that prevent the corresponding implications, in the context of this text, even if we restrict $f$ to be real valued and bounded.

\begin{figure}[!]
    \centering
    \includegraphics[width=0.9\linewidth]{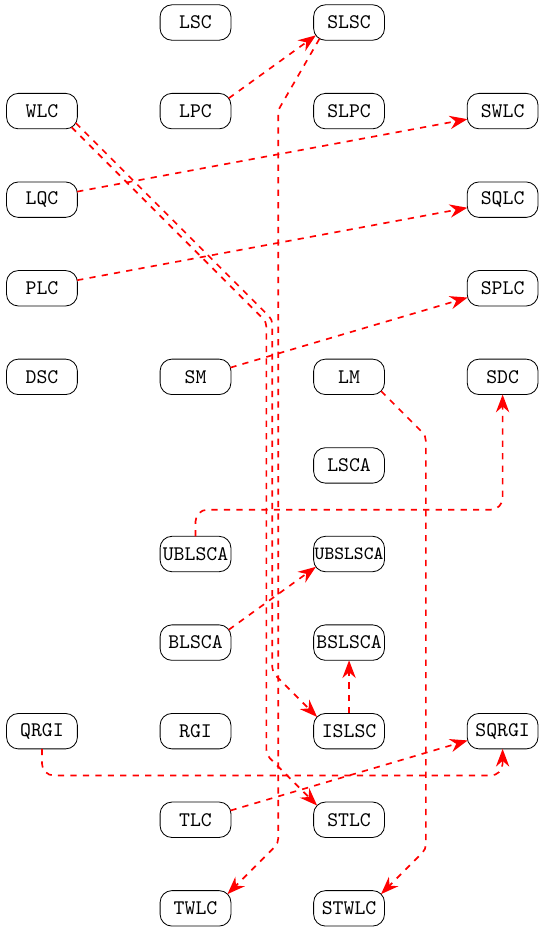}
    \caption{Arrows with counter examples preventing inference of the corresponding implication.}
    \label{fig:counter_examples}
\end{figure}

\noindent $(\mathtt{SLSC})\nRightarrow(\mathtt{TWLC})$\\ 
From \cite[ex 2.3]{morgan2004pseudocontinuity}, consider $\mathcal X = (0,1)$ equipped with the countable complement topology and take
\[f(x) = x.\] 

\vspace{\baselineskip}

\noindent $(\mathtt{WLC})\nRightarrow(\mathtt{STLC})$ and $(\mathtt{WLC})\nRightarrow(\mathtt{ISLSC})$\\ 
    Take $\mathcal X=\mathbb R$ with the standard Euclidean topology and let 
    \[f(x) = \begin{cases}
        1 & x\geq 0,\\
        0 & x<0.
    \end{cases}\] 
    For each $x\in\mathcal{X}$, if there is any $y$ such that $f(y)<f(x)$, then
    $y<0$ with
    \begin{gather*}
    f(y)=0=\inf_{\mathcal{X}}f\leq \inf_U f,\quad\mbox{for any }U\in\mathcal{N}(x).
    \end{gather*}
    $f$ is then \texttt{WLC} at $x$. 
    However, $f$ is not \texttt{STLC} at 0 as the sequence $x_n=-\frac1n\to 0$ is such that for any $y\in\mathcal X$, 
    \[f(y)\geq f(x_n)\qquad\forall n\in\mathbb N\text.\]
    $f$ is also not \texttt{ISLCS} at $0$ since $(x_n)$ is a minimizing sequence converging
    to $x$. 

    \vspace{\baselineskip}

    \noindent
    $(\mathtt{LQC})\nRightarrow (\mathtt{SWLC})$\\ 
    Let $\mathcal X=[-1,2]$ be equipped with its standard topology and let 
    \[f(x) = \begin{cases}
        -x, & -1\leq x\leq 0,\\
        -x-1, & 0<x\leq 1,\\ 
        x-3, & 1<x\leq 2.
        \end{cases}\]
    $f$ is continuous at all $x\neq 0$ and at $x=0$, $\forall y\in\mathcal X$ with $f(y) < f(x)$, 
    \[f(y) \leq -1 = \liminf_{y\to x} f\text.\] 
    Hence, $f$ is \texttt{LQC}. 
    $f$ is not \texttt{WLC} at $x=0$ because $\forall U\in \mathcal N(x)$, 
    \[\inf_U f < -1\] 
    and by letting $y=2$ we have
    \[0=f(x) > f(y) = -1 > \inf_U f\text.\] 
    Because $\mathcal X$ is $N_1$, this counter example also implies that $(\mathtt{LQC})\nRightarrow (\mathtt{SWLC})$, as required.
    Note that $-f$ was used in \cite[ex 1]{scalzo2009uniform} as a counter example to show that upper quasi-continuity does not imply upper weak continuity.
    
    \vspace{\baselineskip}

    \noindent $(\mathtt{LPC})\nRightarrow(\mathtt{SLSC})$\\ 
    From \cite[ex 4.1]{morgan2007pseudocontinuous}, $(\mathtt{LPC})\nRightarrow(\mathtt{LSC})$ is shown by taking $\mathcal X = [1,3]^2$ with its standard topology, and considering 
    \[f(x,y) = \begin{cases}
        y(x-2) & 1\leq x<2,\\ 
        2 & x=2,\\ 
        x+2 & 2<x\leq 3.
    \end{cases}\] 
    We have the required result because $\mathcal X$ is $N_1$.
    
    We also take the opportunity to present a more elementary example. Let $\mathcal X=(0,1]$ with its standard topology, and consider 
    \[f(x) = \begin{cases}
        x & x<1,\\ 
        2 & x=1.
    \end{cases}\] 
    Then, $f$ is not \texttt{LSC} since $\lev_{\leq 1}f = (0,1)$ is not closed.
    Because $\mathcal X$ is $N_1$, it also follows that $f$ is not \texttt{SLSC}. 
    
    We claim that $f$ is \texttt{LPC}. 
    Indeed $\forall x\neq 1$, $f$ is continuous at $x$ and so it is \texttt{LPC} at $x$. 
    At $x=1$, 
    \[\liminf_{z\to x} f(z) = 1,\] 
    and $\forall y\in \mathcal X$ with $f(y) < f(1)=2$, it holds that $f(y)<1$.
    That is, $\forall y\in\mathcal X$ with $f(y) < f(x)$, it follows that
    \[f(y) < \liminf_{z\to x}f(z),\] 
    and therefore $f$ is \texttt{LPC}, as required. 
    

    \vspace{\baselineskip}
    
    \noindent $(\mathtt{LM})\nRightarrow(\mathtt{STWLC})$\\  

    Let $\mathcal X=(\mathbb R\times \mathbb N)/ (\set 0\times \mathbb N)$ be equipped with the quotient topology for $\mathbb R\times \mathbb N$, equipped with the product topology. 
    For $x\neq 0$ and $n\in\mathbb N$, we can take a neighbourhood base
    \[\mathcal B(x,n) = \set{(x-\epsilon,x+\epsilon)\times \set n:\epsilon>0},\] 
    and for the equivalence class $[(0,1)]$, we can take neighbourhood base
    \[\set{\bigcup_{m\in\mathbb N} (-\epsilon_m,\epsilon_m)\times \set m:\epsilon_m>0}.\] 
    For ease of notation, we write $\mathbf 0 = [(0,1)]$ and for any $y\neq 0$ define the coordinate projections $\pi_1(y,n)=y$ and $\pi_2(y,n)=n$.
    
    We will show that for any sequence $(x_n)\subseteq \mathcal X\setminus \set{\mathbf 0}$ converging to $\mathbf 0$, it holds that

    
    \[|\set{\pi_2(x_n):n\in\mathbb N}| < \aleph_0.\] 
    Indeed, if instead it holds that $|\set{\pi_2(x_n):n\in\mathbb N}| = \aleph_0$, then there exists a subsequence $x_{n_m}$, such that $(\pi_2(x_{n_m}))$ is strictly increasing. Then,
    \[U = \ab{\bigcup_{k\in \set{n_m}} (-|\pi_1(x_k)|,|\pi_1(x_k)|)\times \set k} \cup \ab{\bigcup_{k\notin\set{n_m}} (-1,1)\times \set k}\] 
    is a neighborhood of $\mathbf 0$, but $(x_{n_m})$ is never in $U$.
    In particular $x_{n_m}\nrightarrow\mathbf 0$, and thus $x_n\nrightarrow \mathbf 0$, as required.

    As an immediate consequence, if $(x_n)\subseteq \mathcal X\setminus \set{\mathbf 0}$ converges to $\mathbf 0$, then it is necessary that there exists a subsequence $(x_{n_m})$, such that $(\pi_2(x_{n_m}))$ is constant. 
    Now let $f:\mathcal X\to \mathbb R$ be given by 

    
    \[f\left(y\right)=\begin{cases}
    \frac{1}{n}\text{e}^{-\left|x\right|} & y=(x,n)\in{\cal X}\backslash\left\{ \mathbf{0}\right\} ,\\
    2 & y=\mathbf0.
    \end{cases}\]
    We first claim that $f$ is \texttt{LM}. To this end, observe that
    for any $x\in \mathcal X\setminus \set{\mathbf 0}$, because the topology locally resembles the topology on $\mathbb R$, $f$ is continuous at $x$.  In particular, $f$ is \texttt{LM} at $x$.
    
    Now take any $x_n\to \mathbf 0$.
    If $(x_n)$ is eventually $\mathbf 0$, then trivially 
    \[f(\mathbf 0) \leq \lim_{n\to\infty} f(x_n).\] 
    If instead, by arguments above, there is a subsequence of $(x_{n_m})$ with $(\pi_2(x_{n_m}))$ constant.
    Along $\mathbb R\times \set{\pi_2(x_{n_m})}$, the topology locally resembles the topology on $\mathbb R$, thus it is impossible for $x_{n_m}\to \mathbf 0$ with $(f(x_{n_m}))$ decreasing.
    It is then impossible for $(f(x_n))$ to be decreasing and thus $f$ is vacuously \texttt{LM} at $\mathbf 0$. 

    We now claim that $f$ is not \texttt{STWLC} at $\mathbf 0$. 
    Indeed, take any $y\in\mathcal X$ and choose any $n\in\mathbb N$ with $\frac 1n< f(y)$, where we note that such an $n$ exists.  
    Then the sequence defined by $x_m = (1/m,n)$ converges to $\mathbf 0$, but 
    \[\liminf_{m\to\infty} f(x_m) = \frac 1n < f(y).\] 
    It follows that $f$ is not \texttt{STWLC}, as required. 
    
    \vspace{\baselineskip}

    \noindent $(\mathtt{UBLSCA})\nRightarrow(\mathtt{SDSC})$\\ 
    Let $\mathcal X=[-1,1]$ with its standard topology, and consider 
    \[f(x) = \begin{cases}
        -1 & x\leq 0,\\ 
        1-x & 0<x<1,\\ 
        1 & x=1.
    \end{cases}\] 
    $f$ is \texttt{UBLSCA} by taking $a=-\frac12$ in the definition of \texttt{UBLSCA}. However,
    $f$ is not \texttt{SDSC} at $x=1$, because $x_n=1-\frac1n$ converges to $x$, with $(f(x_n))$ strictly decreasing, but 
    \[\lim_{n\to\infty} f(x_n) =0 < 1 = f(x).\] 

    \vspace{\baselineskip}

    \noindent $(\mathtt{BLSCA})\nRightarrow(\mathtt{UBSLSCA})$\\ 
    As a counter example to show $(\mathtt{BSLSCA})\nRightarrow(\mathtt{UBSLSCA})$, \cite[ex 3.7 (b)]{bottaro2010some} propose to take $\mathcal X=\mathbb R$, with its standard topology, and considers

    \[ f(x) = \begin{cases}
        \arctan(x) & x\in \bigcup_{n\in\mathbb N} (-2n-1, -2n],\\ 
        0 & \text{otherwise.}
    \end{cases}\] 
    Because $\mathcal X$ is $N_1$, this  shows the desired result. 
    
    We can also consider a simpler example. Let $\mathcal X=\mathbb R$ be equipped with its standard topology, and consider 
    \[f(x) = \begin{cases}
        0 & x\leq 0,\\ 
        \frac1{n} & x\in (n-1,n]\text{, }n\in\mathbb N.
    \end{cases}\] 
    To show that $f$ is \texttt{RGI} and thus \texttt{BLSCA}, observe that
    $\forall x\leq 0$, $f(x)=\inf_{\mathcal X}f$, and thus $f$ is \texttt{RGI} at $x$. 
    For any $x>0$, take $U=(0, x+1)\in\mathcal N(x)$.
    Then $\inf_Uf>\inf_{\mathcal X}f=0$, and thus $f$ is \texttt{RGI} at $x$. 
    
    Now, to show that $f$ is not \texttt{UBSLSCA}, 
    take any $a>\inf_{\mathcal X}f = 0$ and  any $N\in\mathbb N$, such that $\frac1{N+1} < a$, and
    consider the sequence, $x_n = N+\frac1n$.
    Then, 
    \[f(x_n) = \frac1{N+1} < a\text.\] 
    $(f(x_n))$ is then decreasing and then have $x_n\to N$ with 
    \[\lim_{n\to\infty} f(x_n) = \frac1{N+1} < \frac1N = f(N),\] 
    and thus $f$ is not \texttt{UBLSCA}.
    Because $\mathcal X$ is $N_1$, $f$ is not \texttt{UBSLSCA}, as required. 

    \vspace{\baselineskip}
    
    \noindent $(\mathtt{ISLSC})\nRightarrow(\mathtt{BSLSCA})$\\ 
    A counter example is given in \cite[ex 3.7(c)]{bottaro2010some}. Here, we provide a simplified version of the argument. Let $\mathcal X = \ell^2(\mathbb R)$ endowed with the weak topology, and let $(e_n)$ denote its standard Shauder basis. Consider 
    \[f(x) = \begin{cases}
        \frac1{\Vert x\Vert} & x\in \bigcup_{n\in\mathbb N} \mathrm{span}(e_n)\setminus \set0,\\ 
        1 & \text{otherwise.}
    \end{cases}\] 
    We will first show that $f$ is \texttt{ISLSC} by demonstrating that $f$ does not admit a converging minimizing sequence.
    To obtain a contradiction, assume that such a sequence exists.
    Denote it by $(x_n)$ and let $x$ be its limit.
    It must then follow that for $n$ sufficiently large, $x_n = \lambda_n e_{m_n}$, for some $\lambda_n \in \mathbb R\setminus \set0$, with 
    \[|\lambda_n|\to \infty.\] 
    We then have $\Vert x_n\Vert = |\lambda_n|$, and thus $(x_n)$ is unbounded in norm.
    However, weakly convergent sequences are bounded; thus, we have the required contradiction. 

    We will show that $f$ is not \texttt{BSLSCA} at $x=0$. 
    Take any $a>\inf_{\mathcal X}f=0$, and  any $M\in\mathbb [0,\infty)$ with 
    \[\frac1M < \min\set{\frac12, a}.\] 
    Upon letting $x_n = Me_n$, it then follows that  
    \[f(x_n) = \frac1M < \min\set{\frac12, a}.\] 
    In particular, $\forall n\in\mathbb N$, we have $f(x_n) \leq a$ and 
    \[\lim_{n\to\infty} f(x_n) < 1 = f(x).\] 
    The proof is complete if we show that $(x_n)$ converges weakly to $x=0$.
    Take any $y\in \ell^2(\mathbb R)$ and write 
    \[y = \sum_{n=1}^\infty a_n e_n.\] 
    By definition of $\ell^2$,
    \[\sum_{n=1}^\infty |a_n|^2 < \infty,\] 
    and thus, it must hold that $a_n\to 0$. 
    It follows that 
    \[\inprod{y}{x_n} = Ma_n \to 0 = \inprod y0,\] 
    therefore $(x_n)$ converges weakly to $x$, as required.

   \vspace{\baselineskip}
   
   \noindent $(\mathtt{TLC})\nRightarrow(\mathtt{SQRGI})$\\ 
    Let $\mathcal X=[-1,1]$ with its standard topology, and consider 
    \[f(x) = \begin{cases}
        0 & x\leq 0,\\ 
        1-x & 0<x<1,\\ 
        1 & x=1.
    \end{cases}\] 
    Then, $f$ is continuous on $\mathcal X\setminus \set{0,1}$, and in particular, it is also \texttt{TLC} on the set. Further, 
    $f$ is \texttt{LSC} at $x=0$, thus also \texttt{TLC}, and
    at $x=1$, letting $U=(1/2,1]\in \mathcal N(x)$, it follows that
    $\forall z\in U$, $f(z) > f(0)$, therefore 
    $f$ is also \texttt{TLC} at $x=1$.

    We will show that $f$ is not \texttt{QRGI} at $x=1$. Because $\mathcal X$ is $N_1$, it will follow that $f$ is not \texttt{SQRGI} at $x=1$, as required.
    Observe that $\forall U\in\mathcal N(1)$, $\inf_Uf = \inf_{\mathcal X}f$. 
    Additionally, for $U= (1/2,1]\in \mathcal N(1)$, we have that $\inf_{\mathcal X}f \notin f(U)$, and thus
    neither of the conditions required for $f$ to be \texttt{QRGI} are satisfied, yielding the counter example.

   \vspace{\baselineskip}

    %

    \noindent $(\mathtt{PLC})\nRightarrow(\mathtt{SLQC})$\\ 
    Let $\mathcal X=[0,1]$ be equipped with its standard topology, and consider 
    \[f(x) = \begin{cases}
        2 & x=0,\\ 
        0 & x\in(0,1),\\ 
        1 & x=1.
    \end{cases}\] 
    To show that $f$ is \texttt{PLC}, observe that for $x\in(0,1)$, $f(x) = \inf_{\mathcal X}f$ and so $f$ is vacuously \texttt{PLC} at $x$. 
    At $x=1$, there is no $y$ such that $f(y) < f(x)$ and $(x,y)$ is not a jump point. Hence, $f$ is again vacuously \texttt{PLC} at $x$. 
    At $x=0$, the only $y$ with $f(y) < f(x)$ and $(x,y)$ not a jump point have $f(y) = 0 = \inf_{\mathcal X}f$. Hence, $f$ is \texttt{PLC} at $x$.
    This accounts for all possible $x$, and so $f$ is \texttt{PLC}.
    
    However, $f$ is not \texttt{LQC} at $x=0$ because 
    \[\liminf_{z\to 0} f(z) = 0,\] 
    but 
    \[0<f(1)< f(0).\] 
    Because $\mathcal X$ is $N_1$, $f$ is also not \texttt{SLQC} as required.

   \vspace{\baselineskip}

    \noindent
    $(\mathtt{SM})\nRightarrow (\mathtt{SPLC})$\\
     Let $\mathcal X=[0,3]$ be equipped with its standard topology, and consider 
    \[f(x) = \begin{cases}
        x & x\leq 2,\\ 
        3-x & \text{otherwise.}
    \end{cases}\] 
     Observe that $\forall x\neq 2$, $f$ is continuous at $x$ and so $f$ is \texttt{SM} at $x$, and at $x=2$ there is no non-trivial net $(x_\alpha)$, with $x_\alpha \to x$ and $(f(x_\alpha))$ is decreasing.
     Therefore, $f$ is \texttt{SM}. 

    However, since 
    \[\liminf_{z\to 2}f(z) = 1\] 
    but 
    \[1 < f(3/2) = \frac{3}{2} < f(2) = 2.\]
    Because $f$ has no jump points, this shows that $f$ is not \texttt{PLC} at $x=2$, and because $\mathcal X$ is $N_1$, this then implies that $f$ is not \texttt{SPLC}.

    \vspace{\baselineskip}

    \noindent
    $(\mathtt{QRGI})\nRightarrow (\mathtt{SQRGI})$\\ 
    Let $A$ denote $\mathbb R$ equipped with the countable complement topology, and let $B$ denote $\mathbb R$ with its usual topology.
    Let $\mathcal X=A\times B$ be equipped with the product topology and
    consider 
    \[f(x,y) = \begin{cases}
        1 & x=0,\, y=0,\\
        |y| & \text{otherwise.}
    \end{cases}\] 
    To show that $f$ is \texttt{QRGI},
    take any $(x,y)\in\mathcal X$.
    If $y>0$, then $U:=A\times (y/2,\infty)\in\mathcal N(x,y)$ with 
    $$\inf_{\mathcal Y}f=y/2>0=\inf_{\mathcal X}f,$$
    and thus, $f$ is \texttt{QRGI} at these points. Following an analogous argument, $f$ also \texttt{QRGI} at all points where $y<0$.
    Next, for $y=0$, if $x\neq 0$,then $f(x,y) = \inf_{\mathcal X}f$. Thus $f$ is \texttt{LSC} at $(x,y)$ and so in particular \texttt{QRGI} at $(x,y)$. 
    When $(x,y)=(0,0)$, any basic neighbourhood $U\times V$ of $(x,y)$ has $U$ with countable complement. 
    Therefore, $\inf_{\mathcal X}f = 0 \in f(U\times V)$.
    Then, because any neighbourhood $W$ of $(0,0)$ contains a basic neighbourhood, we obtain that $\inf_{\mathcal X}f\in f(W)$, implying that $f$ is \texttt{QRGI} at $(x,y)=(0,0)$.

    Now, we show that $f$ is not \texttt{SQRGI} at $(x,y)=(0,0)$. To this end,
    observe that the sequence $(x_n,y_n) = (0,1/n)$ is minimizing and converges to $(0,0)$. 
    We are then done if there does not exist a sequence $(z_n)\subset\argmin_{\mathcal X}f$ with $z_n\to (0,0)$.
    Any sequence $(x_n,y_n)\subset\argmin_{\mathcal X}f$ must have $y_n=0$ and $x_n\neq 0$.
    Thus, such a sequence is in $A\times \set 0$.
    Hence, for $(x_n,y_n)$ to converge to $(0,0)$, $(x_n)$ must be eventually constant at $0$.
    However, $\forall n\in\mathbb N$, $x_n\neq 0$ and so no $(x_n,y_n)\subset \argmin_{\mathcal X}f$ can converge to $(0,0)$. 
    $f$ is then not \texttt{SQRGI}, as required.

\section*{Appendix}
Below, we collect references regarding the known implications and counter examples presented in \cref{fig:known}.
\begin{enumerate}
    \item If $f$ is \texttt{LSC}, then it is \texttt{LPC}: \cite[p 175]{morgan2007pseudocontinuous}.
    \item There exists an $f$ that is \texttt{LPC} but not \texttt{LSC}: \cite[ex 4.1]{morgan2007pseudocontinuous}.
    \item If $f$ is \texttt{LSC}, then $f$ is \texttt{LSCA} \cite[ex 1.3]{chen2002note}.
    \item There exists an $f$ that is \texttt{LSCA} but not \texttt{LSC}:  \cite[ex 1.3]{chen2002note}
    \item If $f$ is \texttt{SLSC} then $f$ is \texttt{SLPC}: \cite[rem 2.1]{morgan2004pseudocontinuity}.
    \item $f$ There exists an $f$ that is \texttt{SLPC} but not \texttt{SLSC}: \cite[rem 2.1]{morgan2004pseudocontinuity}.
    \item $f$ There exists an $f$ that is \texttt{SLSC} but not \texttt{TWLC}: \cite[ex 2.3]{morgan2004pseudocontinuity}.
    \item If $f$ is \texttt{LPC} at $x$, then $f$ is \texttt{RGI} at $x$: \cite[prop 7]{amini2016some}.
    \item If $f$ is \texttt{LPC}, then $f$ is \texttt{SLPC}: \cite[prop 2.3]{morgan2007pseudocontinuous}.
    \item If $\mathcal X$ is $N_1$ and $f$ is \texttt{SLPC}, then $f$ is \texttt{LPC}: \cite[prop 2.3]{morgan2007pseudocontinuous}.
    \item If $f$ is \texttt{LPC}, then $f$ is \texttt{TLC}: \cite[p.g. 175]{morgan2007pseudocontinuous}.
    \item There exists an $f$ that is\texttt{TLC} but not \texttt{LPC}: \cite[p 175]{morgan2007pseudocontinuous}.
    \item There exists an $f$ that is \texttt{SLPC} but not \texttt{LPC}: \cite[prop 2.3]{morgan2007pseudocontinuous}.
    \item If $f$ is \texttt{LPC}, then $f$ is \texttt{WLC}: \cite{scalzo2009uniform}.
    \item If $f$ is \texttt{SLPC}, then $f$ is \texttt{SLQC}: \cite[rem 3.1]{morgan2004new}.
    \item If $f$ is \texttt{SLPC}, then $f$ is \texttt{ISLSC}: \cite[p 317(a)]{aruffo2008generalizations}.
    \item There exists an $f$ that is \texttt{SLQC} but not \texttt{SLPC}: \cite{morgan2006discontinuous}. 
    \item If $f$ is \texttt{WLC}, then $f$ is \texttt{LQC}: \cite{scalzo2009uniform}.
    \item There exists an $f$ that is \texttt{LQC} but not \texttt{WLC}: \cite[ex 1]{scalzo2009uniform}.
    \item If $f$ is \texttt{LQC} at $x$, then $f$ is \texttt{QRGI} at $x$: \cite[prop 7]{amini2016some}.
    \item There exists an $f$ that is \texttt{QRGI} but not \texttt{LQC}: \cite[ex 8]{amini2016some}.
    \item There exists an $f$ that is\texttt{LQC} but not \texttt{RGI}: \cite[ex 10]{amini2016some}.
    \item If $f$ is \texttt{LQC}, then $f$ is \texttt{SLQC}: \cite[prop 1]{scalzo2009uniform}.
    \item If $\mathcal X$ is $N_1$ and $f$ is \texttt{SLQC}, then $f$ is \texttt{LQC}: \cite[prop 1]{scalzo2009uniform}.
    \item There exists an $f$ that is \texttt{SLQC} but not \texttt{ISLSC}: \cite[p 317(c)]{aruffo2008generalizations}.
    \item There exists an $f$ that is \texttt{ISLSC} but not \texttt{SLQC}: \cite[p 317(c)]{aruffo2008generalizations}.
    \item If $f$ is \texttt{SLQC}, then $f$ is \texttt{STWLC}: \cite[rem 3.1]{morgan2004new}.
    \item If $f$ is \texttt{LSCA}, then $f$ is \texttt{UBSLSCA}: \cite[p 5]{bottaro2010some}.
    \item If $f$ is \texttt{LSCA}, then $f$ is \texttt{SDSC}: \cite{bao2022exact}.
    \item There exists an $f$ that is \texttt{SDSC} but not \texttt{LSCA}: \cite{bao2022exact}.
    \item There exists an $f$ that is \texttt{UBSLSCA} but not \texttt{LSCA}: \cite[ex 3.7(a)]{bottaro2010some}.
    \item If $f$ is \texttt{UBSLSCA}, then $f$ is \texttt{BSLSCA}: \cite[p 5]{bottaro2010some}.
    \item There exists an $f$ that is \texttt{BSLSCA} but not \texttt{UBSLSCA}: \cite[ex 3.7(b)]{bottaro2010some}.
    \item If $f$ is \texttt{BSLSCA}, then $f$ is \texttt{ISLSC}: \cite[p.g. 5]{bottaro2010some}.
    \item If $\mathcal X$ is $N_1$ and $f$ is \texttt{ISLSC}, then $f$ is \texttt{BSLSCA}: \cite[thm 3.4]{bottaro2010some}.
    \item There exists an $f$ that is \texttt{ISLSC} but not \texttt{BSLSCA}: \cite[ex 3.7(c)]{bottaro2010some}.
    \item There exists an $f$ that is \texttt{QRGI} but not \texttt{RGI}: \cite[ex 5]{amini2016some}.
    \item If $f$ is \texttt{RGI}, then $f$ is \texttt{QRGI}: \cite{amini2016some}.
    \item If $f$ is \texttt{RGI} at $x$, then $f$ is \texttt{TLC} at $x$: \cite[prop 7]{amini2016some}.
    \item There exists an $f$ that is \texttt{TLC} but not \texttt{RGI}: \cite[ex 10]{amini2016some}.
    \item If $f$ is \texttt{QRGI} at $x$, then $f$ is \texttt{TWLC} at $x$: \cite[prop 7]{amini2016some}.
    \item There exists an $f$ that is \texttt{TWLC} but not \texttt{QRGI}: \cite[ex 9]{amini2016some}.
    \item If $f$ is \texttt{LSCA}, then $f$ is \texttt{ISLSC}: \cite[rem 4.2]{aruffo2008generalizations}.
    \item If $\mathcal X$ is $N_1$ and $f$ is \texttt{ISLSC}, then $f$ is \texttt{STWLC} \cite[p 317(d)]{aruffo2008generalizations}. 
    \item There exists an $f$ that \texttt{ISLSC} but not \texttt{STWLC}: \cite[p 317(e)]{aruffo2008generalizations}.
    \item If $f$ is \texttt{TLC}, then $f$ is \texttt{TWLC}: \cite{tian1995transfer}.
    \item There exists an $f$ that is \texttt{TWLC} but not \texttt{TLC}: \cite[ex 2]{tian1995transfer}.
    \item If $f$ is \texttt{TWLC}, then $f$ is \texttt{STWLC}: \cite[prop 2.1]{morgan2004new}.
    \item There exists an $f$ that is \texttt{STWLC} but not \texttt{TWLC}: \cite[prop 2.1]{morgan2004new}.
    \item If $\mathcal X$ is $N_1$ and $f$ is \texttt{STWLC}, then $f$ is \texttt{TWLC}: \cite[prop 2.1]{morgan2004new}.
    \item There exists an $f$ that is \texttt{TWLC} but not \texttt{LQC}: \cite{morgan2006discontinuous}.
\end{enumerate}

\bibliographystyle{siamplain}
\bibliography{refs}
\end{document}